\documentclass{article}

\usepackage[margin=1.1in]{geometry}
\usepackage[utf8]{inputenc} 
\usepackage[T1]{fontenc}    
\usepackage[backref=page]{hyperref}
\usepackage{url}            
\usepackage{booktabs}       
\usepackage{amsfonts}       
\usepackage{nicefrac}       
\usepackage{microtype}      
\usepackage{tikz}

\usepackage[inline]{enumitem}   

\usepackage[nottoc,numbib]{tocbibind}

\usetikzlibrary{shapes,arrows,positioning,arrows.meta}

\tikzstyle{block} = [rectangle, draw,
    text width=16em, text centered, rounded corners, minimum height=4em]
\tikzstyle{block2} = [rectangle, draw,
    text width=8em, text centered, rounded corners, minimum height=2.5em]
\tikzstyle{block3} = [rectangle,
    text width=8em, text centered, minimum height=4em]
\tikzstyle{line} = [draw, -implies, double distance=3pt]
\tikzstyle{dashedline} = [draw, dashed]
\tikzstyle{doubleimplies} = [draw, implies-implies, double distance=3pt]
\tikzstyle{arrow} = [draw, -{Latex[length=3mm]}]
\tikzstyle{doublearrow} = [draw, {Latex[length=3mm]}-{Latex[length=3mm]}]
\tikzstyle{noblock} = [rectangle, text width=5em, text centered,
    minimum height=4em]
\tikzstyle{widenoblock} = [rectangle, text width=9em, text centered,
    minimum height=2em]
\tikzstyle{wideblock} = [rectangle, draw,
    text width=20em, text centered, rounded corners, minimum height=6em]
\tikzstyle{widesmallblock} = [rectangle, draw,
    text width=15em, text centered, rounded corners, minimum height=1.5em]

\usepackage{amsmath,amsfonts,amsthm,amssymb}
\usepackage[]{xcolor}
\usepackage[numbers]{natbib}
\usepackage[normalem]{ulem}

\newcommand{\Pval}{\mathfrak p}

\newcommand{\ML}{\textnormal{ML}}
\newcommand{\JP}{\textnormal{JP}}
\newcommand{\NA}{\textnormal{NA}}

\newcommand{\RR}{\mathbb{R}}
\newcommand{\NN}{\mathbb{N}}
\newcommand{\EE}{\mathtt{E}}

\newcommand{\Pbb}{\mathcal{P}}

\newcommand{\Ptt}{\mathtt{P}}
\newcommand{\Qtt}{\mathtt{Q}}
\newcommand{\Rtt}{\mathtt{R}}

\newcommand{\1}{\mathbf{1}}
\newcommand{\dd}{\mathrm{d}}

\newcommand{\Xcal}{\mathcal{X}}

\newcommand{\Fcal}{\mathcal{F}}
\newcommand{\Acal}{\mathcal{A}}

\newcommand{\Pcal}{\mathcal{P}}

\newcommand{\Qcal}{\mathcal{Q}}

\newcommand{\Gcal}{\mathcal{G}}

\newcommand{\half}{\tfrac{1}{2}}

\DeclareMathOperator*{\esssup}{ess\,sup}

\newtheorem{theorem}{Theorem}
\newtheorem{lemma}[theorem]{Lemma}
\newtheorem{proposition}[theorem]{Proposition}
\newtheorem{example}[theorem]{Example}
\newtheorem{remark}[theorem]{Remark}

\newtheorem{definition}[theorem]{Definition}
\newtheorem{corollary}[theorem]{Corollary}

\usepackage{graphicx} 
\usepackage{caption}
\usepackage{subcaption}
\usepackage{enumitem}
\usepackage{comment}

\usepackage{setspace}

\hypersetup{
     colorlinks   = true,
     citecolor    = gray
}
\hypersetup{linkcolor=black}

\title{
Testing exchangeability:\\
fork-convexity, supermartingales, and e-processes
\bigskip
}

\author{
Aaditya Ramdas$^{1}$,  Johannes Ruf$^{2}$, Martin Larsson$^{3}$,  Wouter M. Koolen$^{4}$\\
\and 
$^{1}$ Departments of Statistics and Machine Learning, Carnegie Mellon University \\
$^{2}$ Department of Mathematics, London School of Economics\\
$^{3}$ Department of Mathematical Sciences, Carnegie Mellon University \\
$^{4}$ Machine Learning Group, CWI Amsterdam \\
\and
 \texttt{\{aramdas,martinl\}@andrew.cmu.edu}\\
 \texttt{wmkoolen@cwi.nl, j.ruf@lse.ac.uk}
}

\begin{document}

\maketitle

\begin{abstract}
Suppose we observe an infinite series of coin flips $X_1,X_2,\ldots$, and wish to  sequentially test the null that these binary random variables are exchangeable. Nonnegative supermartingales (NSMs) are a workhorse of sequential inference, but we prove that they are powerless for this problem. First, 
utilizing a geometric concept called fork-convexity (a sequential analog of convexity),
we show that any process that is an NSM under a set of distributions, is also necessarily an NSM under their ``fork-convex hull''.
Second, we demonstrate that the fork-convex hull of the exchangeable null consists of all possible laws over binary sequences; this implies that any NSM under exchangeability is necessarily nonincreasing, hence always yields a powerless test for any alternative. 
Since testing arbitrary deviations from exchangeability is information theoretically impossible, we focus on Markovian alternatives. 
We combine ideas from universal inference 
and the method of mixtures 
to derive a ``safe e-process'', which is a nonnegative process with expectation at most one under the null at any stopping time, and is upper bounded by a martingale, but is not itself an NSM. This in turn yields a level $\alpha$ sequential test that is consistent; regret bounds from universal coding also demonstrate rate-optimal power. We present ways to extend these results to any finite alphabet and to Markovian alternatives of any order using a ``double mixture'' approach. We provide an array of simulations, and give general approaches based on betting for unstructured or ill-specified alternatives. Finally, inspired by Shafer, Vovk, and Ville, we provide game-theoretic interpretations of our e-processes and pathwise results.
\end{abstract}


 \newpage

{\tableofcontents}

 \newpage

\hypersetup{linkcolor=blue}

\section{Testing exchangeability}

Suppose we observe a sequence of binary coin flips $X_1,X_2,\ldots$ one at a time. Consider the fundamental problem of testing if our data $(X_t)_{t \geq 1}$ form an exchangeable sequence: 
\[
    H_0: X_1,X_2,\ldots \text{ are exchangeable.}
\]
For the unfamiliar reader, $(X_t)_{t \geq 1}$ is exchangeable if and only if for every $t$ and every permutation $\sigma$ of the first $t$ indices, $(X_1,\ldots,X_t)$ has the same distribution as $(X_{\sigma(1)},\ldots,X_{\sigma(t)})$.
If we find enough evidence against the null, we would like to stop collecting data and reject the null as soon as possible. 
Let $\overline \Qcal$ represent the set of all exchangeable distributions over infinite binary sequences. Then, we can rephrase the null as $H_0: \text{the data are generated from some } \Qtt \in \overline \Qcal$.

Let $(\Fcal_t)_{t \geq 0}$ represent the canonical filtration of the data, where $\Fcal_0$ is the trivial sigma algebra and $\Fcal_t = \sigma(X_1,\dots,X_t)$. All martingale statements in this paper will implicitly refer to this canonical filtration. A level $\alpha$ sequential test for $\overline \Qcal$ (that is, for $H_0$) is any stopping time $\tau_\alpha$ such that
\begin{equation}\label{eq:type-1}
\sup_{\Qtt \in \overline \Qcal}~ \Qtt(\tau_\alpha < \infty) \leq \alpha,
\end{equation}
meaning that with probability $1-\alpha$, we never stop under the null. In this paper, we will design powerful sequential tests for exchangeability, and prove their consistency, and in some cases rate-optimality, against Markovian alternatives. 

Despite nonnegative supermartingales (NSMs) being a central object in sequential testing~\citep{howard_exponential_2018,howard_uniform_2019}, the nontriviality of our contribution stems from the following: Theorem~\ref{thm:every-law-fork-convex} proves that \emph{every NSM under exchangeability must be a nonincreasing process and thus powerless}. The way out is to look beyond NSMs, at an object called a safe e-process, which we introduce next. Safe e-processes are intimately related to NSMs --- all NSMs are safe e-processes, and all admissible safe e-processes are infima of NSMs~\cite{ramdas2020admissible} --- but the two concepts are different, and our paper seeks to explore this gap in detail and exploit it successfully.

\subsection{Safe e-processes and anytime-valid p-processes}

We define a $\overline \Qcal$-safe e-process to be a nonnegative sequence of adapted random variables $(E_t)_{t \geq 0}$ such that
\begin{equation}\label{eq:e-process}
\sup_{\Qtt \in \overline  \Qcal} ~ \sup_{\tau} ~ \EE_\Qtt [E_\tau] \leq 1,
\end{equation}
where the second supremum is over all stopping times $\tau$, possibly infinite.
Above, we interpret $E_\infty := \limsup_{t \to \infty} E_t$ to accommodate potentially infinite stopping times.
(As mentioned earlier, the filtration $(\Fcal_t)$ defined previously is fixed and implicit.)
\citet{howard_exponential_2018} do not use the term e-process, but they study exponential processes that are upper bounded by composite nonnegative supermartingales, which is a characterizing property of e-processes.
Related notions have been studied by \citet{shafer_test_2011,grunwald_safe_2019,ramdas2020admissible,vovk2019values}, amongst others.

\begin{remark}[on terminology]
An `e-variable' is a nonnegative random variable that has expectation at most one under the null~\cite{vovk2019values}; in other words $E_t$ is an e-variable for every $t$, while we call $(E_t)$ as an e-process due to the additional property~\eqref{eq:e-process} that insists that stopped e-processes are also e-variables, meaning that $E_\tau$ is an e-variable for any stopping time $\tau$. The value realized by an e-variable will be called an e-value. While recent works in this rapidly evolving area have sometimes blurred the distinctions between e-processes, e-variables and e-values, we believe that our choices of terminology may be more suitable to emphasize sequential problem setups involving stopping times.
\end{remark}

Large e-values encode evidence against the null, and it is easy to check that the stopping time
\begin{equation}\label{eq:stopping-safe-to-seq}
\kappa_\alpha := \inf\left\{t \geq 1: E_t \geq \frac{1}{\alpha}\right\}
\end{equation}
results in a level $\alpha$ sequential test by applying Markov's inequality to the stopped e-process. More details can be found in~\citep{ramdas2020admissible}, who also show that the sequence $(\Pval_t)$ defined by $\Pval_t := \inf_{s \leq t} 1/E_s$ is an anytime-valid (or $\overline \Qcal$-valid) p-process, meaning that
\begin{equation}\label{eq:pvalue}
\sup_{\Qtt \in \overline  \Qcal} ~ \sup_{\tau} ~ \Qtt(\Pval_\tau \leq \alpha) \leq \alpha, ~ \text{ for all $\alpha \in [0,1]$.}
\end{equation}
Here the second supremum is taken over all stopping times $\tau$.
Following the above remark, a p-variable is a random variable whose distribution is stochastically larger than uniform under the null. Thus, a stopped p-process is a p-variable, and we call its realization as a p-value (a number between 0 and 1).

Since such $\overline \Qcal$-safe e-processes yield both sequential tests and anytime-valid p-processes, we focus on constructing e-processes for the rest of this paper.
As a matter of convention, we always use $\kappa_\alpha$ to denote the above stopping time, that is the one that thresholds a safe e-process at level $1/\alpha$, while $\tau$ denotes a generic stopping time.

\paragraph{Power one and consistency.}
We call an e-process powerful if the corresponding test is powerful, and of course we desire a test that is consistent, meaning that its power goes to one with the sample size. Formally, a level $\alpha$ sequential test $\tau_\alpha$ has asymptotically power one against some family $\Pcal \supset \overline \Qcal$ if
 \begin{equation*}
 \inf_{\Ptt \in \Pcal \backslash \overline \Qcal} \Ptt(\tau_\alpha < \infty) = 1.
 \end{equation*}
 (Henceforth, if we mention some alternative $\Pcal$, it is understood that we desire power against $\Pcal\backslash\overline \Qcal$.)
 
A $\Qcal$-safe e-process $(E_t)$ is said to be consistent, or power one, if its associated sequential test $\kappa_\alpha$ from~\eqref{eq:stopping-safe-to-seq} is asymptotically power one at any level, meaning that
\begin{subequations}
 \begin{equation}\label{eq:power-one-e}
 \text{ for all $\alpha \in (0,1)$, ~}~ \inf_{\Ptt \in \Pcal \backslash \overline \Qcal} \Ptt(\kappa_\alpha < \infty) = 1, 
 \end{equation}
 \text{ or equivalently, }
 \begin{equation}
  \limsup_{t \to \infty} E_t = \infty, ~ \text{$\Ptt$-almost surely,  for every } \Ptt\in\Pcal \backslash \overline \Qcal.
 \end{equation}
 \end{subequations}

\subsection{Convex hulls and de Finetti's theorem}

Let the null set $\Qcal$ consist of 
all product distributions $\mathrm{Ber}(p)^\infty$ for some $p\in[0,1]$. Note that $\Qcal$ is a rich composite class of parametric distributions, whose convex hull is $\overline \Qcal$, which is a result well known as de Finetti's theorem~\cite{kirsch2018elementary}.

Any sequential test for $\Qcal$ is also valid for $\overline \Qcal$, a fact that we record below, proved in~\citep{ramdas2020admissible}.

\begin{proposition}
The properties of type-1 error control~\eqref{eq:type-1} and safety are closed under the convex hull, meaning that any $\Qcal$-safe e-process is also $\overline \Qcal$-safe, and any level $\alpha$ sequential test for $\Qcal$ is also valid for $\overline \Qcal$.
\end{proposition}

As a consequence, we may restrict our attention to developing a $\Qcal$-safe e-process, and invoke the above fact to step from the i.i.d.\ setting to the exchangeable setting. This will be our approach in the rest of this paper.
Another consequence is that testing the null $\Qcal$ against the alternative $\overline \Qcal$ is futile; safe and consistent e-processes do not exist and neither do valid, power-one tests.

\begin{remark}
    To avoid confusion, we note that the convex combination of $\Qtt, \Qtt' \in \Qcal$ must be carefully interpreted. For example, if $\Qtt = \mathrm{Ber}(0.3)^\infty$ and $\Qtt' = \mathrm{Ber}(0.7)^\infty$ then a draw from $(\Qtt+\Qtt')/2$ produces either a sequence with 70$\%$ zeros or a sequence with with 70$\%$ ones, each with probability half, while it produces a sequence with equal number of zeros and ones with probability zero. Contrast this with the fact that a draw from $((\mathrm{Ber}(0.3) + \mathrm{Ber}(0.7)) / 2)^\infty$ is equally likely to produce a zero or a one at every point in time, and thus produces a sequence with equal zeros and ones with probability one. Thus, one must take care to differentiate between $((\mathrm{Ber}(0.3) + \mathrm{Ber}(0.7)) / 2)^\infty$, which is not in the convex hull of $\Qtt$ and $\Qtt'$, and $(\Qtt+\Qtt')/2$, which is. Later, we will see that the former lies in the closed \emph{fork-convex hull} of $\Qtt$ and $\Qtt'$.
\end{remark}

It is impossible to have a powerful test for $\overline 
\Qcal$ against its complement $\overline \Qcal^c$, since the alternative is too rich and consists of too many distributions that are too close to $\overline \Qcal$, meaning that there are too many ways to violate exchangeability. For example, it should be apparent to the reader that if the first coin has bias $p_1$ and every other coin has bias $p \neq p_1$, then the resulting sequence is not exchangeable but we would never be able to reliably detect this deviation. This example relies on ensuring that the information required to detect a deviation from the null is exhausted early on in the sequence. To avoid such pathologies it is necessary to restrict the alternative class in some meaningful way. Markovian alternatives are an attractive choice, balancing the needs of relevant practical motivation, tractable mathematical structure, succinct probabilistic description, and intuitive aesthetic appeal.

\subsection{Time-homogeneous Markovian alternatives}

 We focus on the setting of first-order Markov alternatives $\Pcal_1$, and return to address higher-order (and even more general) alternatives later. To describe $\Pcal_1$ more formally, each $\Ptt \in \Pcal_1$ represents a time-homogeneous first-order Markov process with parameters $p_{1|0}$ and $p_{1|1}$. Here we abbreviate $p_{0|0} = 1-p_{1|0}$ and $p_{0|1} = 1-p_{1|1}$. For arbitrary $k \in \mathbb{N}$, we will also consider $\Pcal_k$, the set of $k$-th order Markov processes.

Let $\mathcal C$ represent the distributions that result in constant sequences $(X_t)$, i.e., where $X_1=X_2=\ldots$. These include, of course, the distributions of deterministic sequences $0^\infty$ and $1^\infty$. The class $\mathcal C$ also includes their mixtures, which are distributions that first set $X_1 \sim \mathrm{Ber}(p)$ for some $p$ and then set $X_t=X_1$ for all $t>1$. In other words, conditional on $X_1$, the remaining observations are i.i.d.~$X_t \sim \mathrm{Ber}(X_1)$.

\begin{figure}
\includegraphics[width=0.9\textwidth]{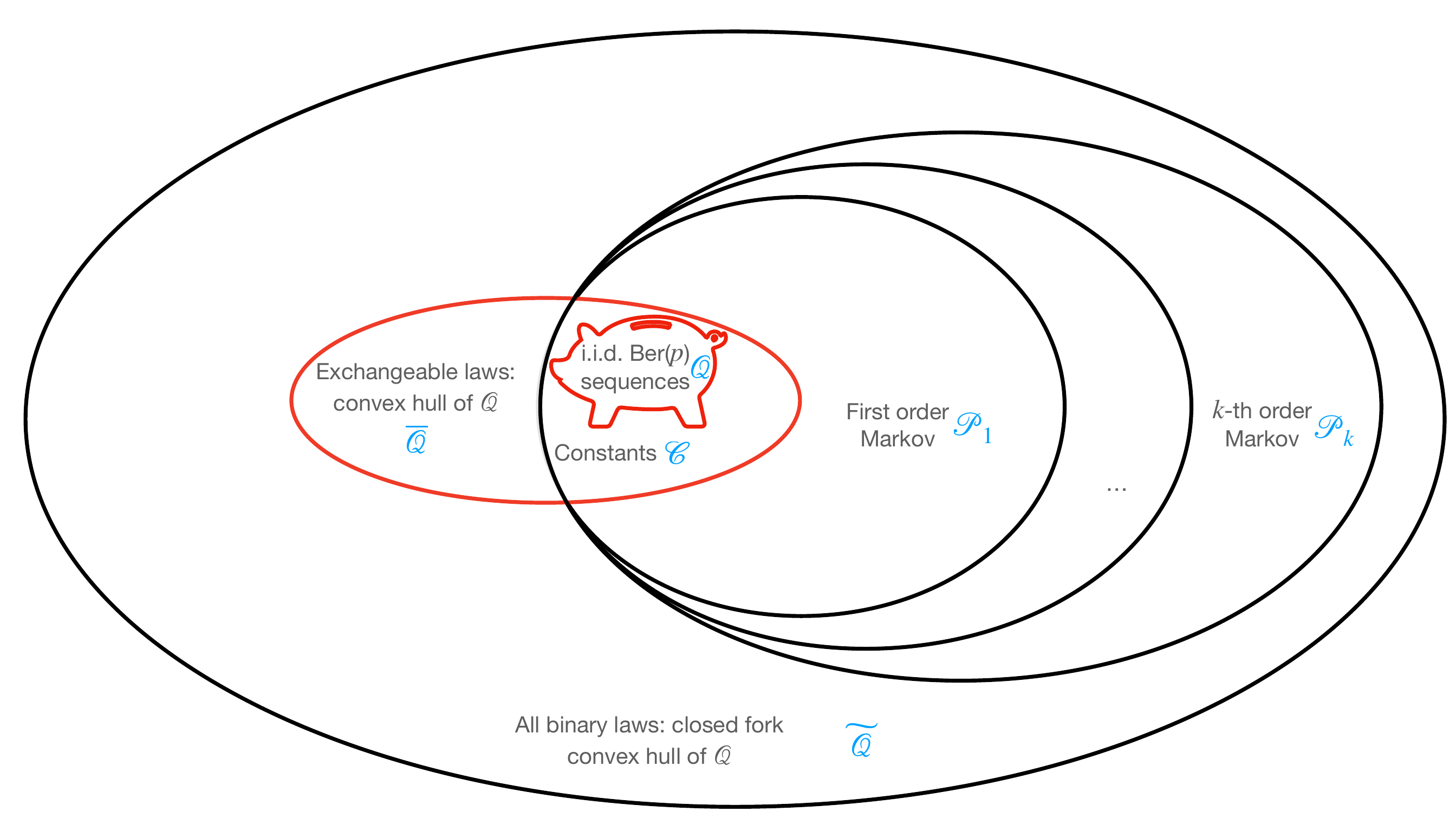}
\caption{Various classes of distributions over infinite binary sequences encountered in this paper. See Proposition~\ref{P:210407} and Theorem~\ref{thm:every-law-fork-convex} for the proof of the displayed set inclusions.}   
\label{fig:nested}
\end{figure}

\begin{proposition} \label{P:210407}
    For any $\{0, 1\}$-valued stochastic process $(X_t)$, the following two conditions are equivalent:
    \begin{enumerate}[label={\rm(\roman{*})}, ref={\rm(\roman{*})}] 
        \item\label{P:210407.1} $(X_t)$ is $k$-th order Markov (with respect to its natural filtration) for some $k \in \mathbb{N}$ and exchangeable.
        \item\label{P:210407.2} $(X_t)$ is either constant (as a function of time) or i.i.d.
    \end{enumerate}
Said differently, $\overline \Qcal \cap \Pcal_k = \Qcal \cup \mathcal C$ for any $k \in \NN$.
\end{proposition}

This relation and more are illustrated in Figure~\ref{fig:nested}. The proof and some additional remarks can be found in Appendix~\ref{A:proofs}.

\subsection{Wald's sequential likelihood (or probability) ratio test}

Let the likelihood under a particular $\Qtt \in \Qcal$, where $\Qtt =\mathrm{Ber}(p)^\infty$, be represented by
\[
Q_t \equiv Q_p(X_1,\dots,X_t) := (1-p)^{n_0} p^{n_1},
\]
where $n_0 = n_0(t)$ and $n_1=n_1(t)$ represent the number of zeros and ones seen up to time $t$. The  likelihood associated to $\Ptt \in \Pcal_1$, say, for simplicity with $\Ptt(X_1=1)=1/2$, is given by
\[
  P_t \equiv P_{p_{1|0}, p_{1|1}}(X_1, \ldots, X_t) ~:=~
  \frac{1}{2}
  \prod_{s = 2}^t p_{X_s|X_{s-1}}
  ~=~
  \frac{
    1
  }{2}
  p_{0|0}^{n_{0|0}}
  p_{1|0}^{n_{1|0}}
  p_{0|1}^{n_{0|1}}
  p_{1|1}^{n_{1|1}}.
\]
Here $n_{1|0}\equiv n_{1|0}(t)$ denotes the total number of ones following zeros up to time $t$, etc. (We omit the dependence on $t$ to ease notational load.)

For a point null $\Qtt \in \Qcal$ and point alternative $\Ptt \in \Pcal_1$ with $\Ptt \neq \Qtt$,  Wald's sequential likelihood (or probability) ratio test (SLRT or SPRT)~\cite{wald_sequential_1945} yields a power-one test. The likelihood ratio process, i.e., $(P_t/Q_t)$, is a $\Qtt$-martingale starting at one and thus a $\Qtt$-safe e-process, and the resulting sequential test $\inf\{t \geq 0: P_t/Q_t \geq 1/\alpha\}$ is a level $\alpha$ test under $\Qtt$ that is known --- up to a discrete-time issue called ``boundary overshoot'' --- to have the smallest expected stopping time under $\Ptt$~\cite{wald_optimum_1948}.

For composite nulls and alternatives, the SLRT cannot be directly applied. The \emph{mixture SLRT} integrates over the alternatives using a ``prior'' distribution (or, more appropriately, mixture distribution, to avoid any Bayesian interpretations of our frequentist statements). However, this only works for composite alternatives, since mixing over the null set does not yield a safe e-process or the desired type-1 error control property in~\eqref{eq:type-1}. (We note here that, interestingly, the GROW e-variables of \cite{grunwald_safe_2019} are ratios of mixtures, though they are safe only for a fixed sample size.) The generalized SLRT maximizes the likelihood under both null and alternative, but this also does not yield a martingale or a safe e-process. In both cases, it is difficult to find a threshold for the resulting process that achieves type-1 error control in~\eqref{eq:type-1}, since the SLRT's choice of $1/\alpha$ does not suffice. For a discussion of techniques and results for this approach we refer to~\cite{kaufmann2018mixture,shin2021glrt}.

\subsection{Nonnegative supermartingales (NSMs) and an impossibility result}

Despite the above apparent difficulties in generalizing the SLRT to yield an e-process, it has been recently established that nonnegative (super)martingales play a fundamental role in the design of admissible sequential tests (and the construction of admissible safe e-processes), even for composite nonparametric nulls~\citep{howard_exponential_2018,ramdas2020admissible}. 
In anticipation of the results to follow, it is useful to set up some relevant notation. In what follows, a process $(M_t)_{t \geq 0}$ will be called a $\Qtt$-NM if it is a nonnegative martingale with initial value one, that is, $(M_t)$ is adapted to $(\Fcal_t)$, $M_0=1$ and $\EE_\Qtt[M_t | \Fcal_s] = M_s \geq 0$ for any $s \leq t$. Such processes are called test martingales by~\citet{shafer_test_2011}. If $(M_t)$ is a $\Qtt$-NM simultaneously for every $\Qtt \in \Qcal$, then we will call it a $\Qcal$-NM.  If the equality above is replaced by an inequality $\leq$, then we will call it a $\Qtt$-NSM or $\Qcal$-NSM (nonnegative supermartingale). 

An appropriate variant of the optional stopping theorem~\cite{durrett_probability:_2017} implies that for any $\Qcal$-NSM $(M_t)$ and any stopping time $\tau$ (potentially infinite), the stopped process has expectation at most one, or in other words
\[
\sup_{\Qtt \in \Qcal} \sup_{\tau} \EE_\Qtt[M_\tau] \leq 1,
\]
where the second supremum is over \emph{all} stopping times $\tau$, potentially infinite \citep[Section 3]{ramdas2020admissible}.
Indeed, for any $\Qcal$-NSM, $M_\infty := \lim_{t \to\infty} M_t$ is a well defined random variable, and $\sup_{\Qtt \in \Qcal} \EE_\Qtt[M_\infty] \leq 1 $.
The correspondence with the definition of a safe e-process (recall \eqref{eq:e-process}) is not coincidental --- to construct a $\Qcal$-safe e-process, it suffices to construct a $\Qcal$-NSM. However, we claim the following.

\begin{proposition}\label{prop:fail-NSM}
Every $\Qcal$-NSM is also a $\Pcal_1$-NSM (recall that $\Qcal$ and $\Pcal_1$ contain all i.i.d.~and first-order Markov distributions respectively). In fact, every $\Qcal$-NSM is also a $\Pcal_k$-NSM for every $k \in \NN$, and is also a $\Pcal$-NSM for \emph{any} set $\Pcal$. In other words, any $\Qcal$-safe e-process with nontrivial power cannot be a $\Qcal$-NSM, since the latter is powerless against any alternative $\Pcal$ by virtue of being a $\Pcal$-NSM.
\end{proposition}
This proposition will follow from Theorem~\ref{thm:every-law-fork-convex}.
This paper is as much about understanding the above negative result, as about providing a positive result. In other words, this result probes at the ``gap'' between a $\Qcal$-safe e-process and a $\Qcal$-NSM. The former is a much weaker property than the latter. While the latter suffices for the former, it is by no means necessary, as recently observed in a more abstract setup~\citep{ramdas2020admissible}. Indeed, a $\Qcal$-NSM  $(N_t)$ satisfies the much stronger ``conditional'' property that
\[
\EE_\Qtt[N_\tau | \Fcal_s] \leq N_{\tau \wedge s} \text{ for every $\Qtt \in \Qcal$ and stopping time $\tau$},
\]
which implies the earlier mentioned optional stopping result. In fact, the above property is satisfied if and only if $(N_t)$ is a $\Qcal$-NSM, but the e-process property \eqref{eq:e-process} can be satisfied even by processes that are upper bounded by  $\Qtt$-NSMs for each $\Qtt \in \Qcal$, but are not themselves $\Qcal$-NSMs. It is exactly this gap that we will exploit going forward.

We provide a geometrical characterization of the above phenomenon: essentially, we will show that the above proposition is true because $\Pcal_k$ lies within the ``fork-convex hull'' of $\Qcal$, and we prove that this hull preserves the NSM property of a process. In fact, Theorem~\ref{thm:every-law-fork-convex} proves that the fork-convex hull of $\Qcal$ is all-encompassing, containing every law over binary sequences.  Thus, a $\Qcal$-NSM yields a powerless test against $\Pcal_k$ since it is automatically and unintentionally safe under the alternative as well as under the null.  Along the way, we will encounter other friends from martingale theory, such as the Snell envelope, which previously has not  played a prominent role in the mathematical treatment of sequential testing.

\subsection{A powerful $\Qcal$-safe e-process that is not a $\Qcal$-NSM}

\citet{ramdas2020admissible} show that any $\Qcal$-safe e-process is dominated by a $\Qcal$-safe e-process $(E_t)$ of the form
\[
E_t := \inf_{\Qtt \in \Qcal} M^\Qtt_t,
\]
where $(M^\Qtt_t)$ is a $\Qtt$-NM. As mentioned before, each limiting variable $M^\Qtt_\infty$ is well defined and has expectation at most one, and thus, $E_\infty := \limsup_{t \to \infty} E_t$ also has expectation at most one. 
The safety property immediately holds at infinite times as well, meaning that $\sup_{\Qtt \in \Qcal} \EE_\Qtt [E_\infty] \leq 1$ and $\sup_{\Qtt \in \Qcal}\sup_{\tau} \EE_\Qtt [E_\tau] \leq 1$,
where the second supremum is over all stopping times $\tau$.
To avoid too much suspense before we get into these 
subtle new concepts, we first present our solution immediately, and delay its derivation to the next section. To this end, define 
\begin{equation}\label{eq:solution}
  R_t ~:=~
  {
    \frac{\Gamma \left(n_{0|0}+\half\right) \Gamma \left(n_{0|1}+\half\right) \Gamma \left(n_{1|0}+\half\right) \Gamma \left(n_{1|1}+\half\right)}{ 2\Gamma\left(\half\right)^4 \Gamma (n_{0|0}+n_{1|0}+1) \Gamma (n_{0|1}+n_{1|1}+1)}
  } \ \Big\slash \ 
  {
   \left( \left(\frac{n_1}{t}\right)^{n_1}
    \left(\frac{n_0}{t}\right)^{n_0}
    \right)
  },
\end{equation}
where $\Gamma$ denotes the usual gamma function. 

\begin{theorem}\label{thm:valid}
The process $(R_t)$ is a $\Qcal$-safe e-process, and thus thresholding it at level $1/\alpha$ yields a level $\alpha$ sequential 
test~$\kappa_\alpha$ for exchangeability. 
\end{theorem}

The first three subsections of Section~\ref{sec:jeffreys+ML} is devoted to the justification of the above theorem (its proof is implicit in its derivation).
Later, Theorem~\ref{T:4} proves that this test has power one for first-order Markov alternatives, i.e., \eqref{eq:power-one-e} holds, and quantifies the rate of growth.
Above, $(R_t)$ is not itself a $\Qcal$-NSM, but is nevertheless upper bounded by a (different!) $\Qtt$-NM for every $\Qtt \in \Qcal$, resulting in it being a $\Qcal$-safe e-process. This idea is enabled by bringing together the method of mixtures (using Jeffreys' prior) for combining the composite alternative, with the maximum likelihood under the composite null. Beyond showing that it has power one, one can quantify that it has rate-optimal power by utilizing a regret bound from universal coding. The next section, among others, proves the above theorem, after which we turn to defining fork-convexity. Then we provide some numerical simulations and conclude with a discussion about this paper's approach compared to other possible approaches to the problem.

\subsection{A summary of this paper's contributions}

This paper proposes a nontrivial and powerful test for exchangeability, a fundamental and easy-to-state problem that exposes rich geometric and probabilistic structure. The most directly related work is that of Vovk~\cite{vovk_testing_randomness_2019}, who proposed a very different approach based on conformal prediction, discussed in Section~\ref{sec:vovk-conformal}.

The problem of testing exchangeability also serves as a test-bed to carefully examine the gap between processes that are composite NSMs, and processes that are composite safe e-processes. The former is a special case of the latter and the latter is a necessary ingredient for the former~\cite{ramdas2020admissible}, but the difference between these concepts has remained somewhat implicit in past work~\cite{howard_exponential_2018}. This work constitutes the first detailed investigation of a problem where nontrivial NSMs do not exist (Proposition~\ref{prop:fail-NSM}, Theorem~\ref{thm:every-law-fork-convex}) but powerful safe e-processes do exist (Theorem~\ref{thm:valid}, Theorem~\ref{T:4}).

In the process, we prove many results that are of independent interest. For example, the NSM property is closed under ``fork-convex combinations'' of distributions (Lemma~\ref{L_fconv_comb_supmg}, Theorem~\ref{thm:fork-safe}). We prove that a composite safe e-process can be improved upon by a composite NSM only if the set of distributions is fork-convex, and this improvement is given by a composite Snell envelope (Theorem~\ref{thm:P-Snell}, Corollary~\ref{cor:safe-fork-closure}).

Section~\ref{sec:sims} examines the empirical behavior of our safe e-process in a simulation suite that includes both basic sanity checks and rather subtle and nontrivial alternatives. Finally, Section~\ref{sec:game-theoretic} suggests a way to interpret our e-process in the language of game-theoretic protocols of Shafer and Vovk~\cite{shafer2019game}.

\section{Jeffreys' mixture meets maximum likelihood}\label{sec:jeffreys+ML}

As briefly mentioned earlier, if the null set was a singleton, say corresponding to $\mathrm{Ber}(p)^\infty$, and the alternative was also a singleton, such as when $(X_t)$ deterministically alternates between 0 and 1, then Wald's sequential likelihood ratio test~\cite{wald_sequential_1945} would immediately yield a solution to the problem at hand.
To elaborate, let $f_p$ denote the probability mass function of a $\mathrm{Ber}(p)$ random variable and let $L_{01}$ denote the likelihood function under the alternative:
\[
L_{01}(X_1,\dots,X_t) := \begin{cases}
1 & \text{ if } (X_1,X_2,\dots) = (0,1,0,1,\dots) ;\\
0 & \text{ otherwise. }
\end{cases}
\]
Then, for any point null (indexed by $p$) define the following likelihood ratio:
\[
R^p_t := \frac{L_{01}(X_1,\dots,X_t)}{\prod_{s = 1}^t f_p(X_s)}, \quad\text{  which equals } \quad
\begin{cases}
\frac1{p^{\lfloor t/2 \rfloor} (1-p)^{\lceil t/2 \rceil}} \text{ with $\mathrm{Ber}(p)^\infty$-probability } p^{\lfloor t/2 \rfloor} (1-p)^{\lceil t/2 \rceil}; \\
0 \text{ with $\mathrm{Ber}(p)^\infty$-probability } 1-p^{\lfloor t/2 \rfloor} (1-p)^{\lceil t/2 \rceil}.
\end{cases}
\]
It is easy to check that $(R^p_t)$ is a $\mathrm{Ber}(p)^\infty$-NM, and thus a $\mathrm{Ber}(p)^\infty$-safe e-process, and $\kappa_\alpha$ from~\eqref{eq:stopping-safe-to-seq} yields a valid level $\alpha$ sequential test that coincides with Wald's original proposal~\cite{wald_sequential_1945}. The question is how to generalize this approach to deal with a composite null and a composite alternative in a computationally tractable and statistically powerful manner. 

The following observation deals  the first blow: the only process that is a nonnegative martingale under every i.i.d.~Bernoulli sequence is one that is constant. In other words, the only $\Qcal$-NM is such that $M_t = 1$ for all $t \in \NN_0$. This obviously results in a powerless test. So we then turn our attention to constructing a $\Qcal$-NSM, or a test supermartingale. Unfortunately this approach is dealt a fatal blow by Proposition~\ref{prop:fail-NSM}. As alluded to in the introduction, we cannot employ mixtures in both numerator and denominator because it violates safety by lowering the denominator too much, and we cannot maximize the likelihood in the numerator and denominator because it violates safety by raising the numerator too much. 

Our proposal combines a suitably chosen mixture in the numerator with maximum likelihood in the denominator, thus avoiding both pitfalls.

\subsection{Dealing with the composite null via maximum likelihood}
To start, let us return to the point alternative described above (alternating 0 and 1), and just handle the composite null using maximum likelihood estimation, as proposed in universal inference~\cite{wasserman2020universal}. To elaborate, observe that 
\[
R^{\ML}_t:=\inf_{p \in [0,1]} R^p_t = 
\frac{\text{likelihood under the point alternative}}{\text{maximum likelihood under the null}}
\] 
is a $\Qcal$-safe $e$-value. Indeed, suppose the data truly comes from $\mathrm{Ber}(p^*)^\infty$ for an unknown $p^*$. Then, it is obvious that
$R^{\ML}_t \leq R_t^{p^*}$, 
where the latter process is a $\mathrm{Ber}(p^*)^\infty$-NM. Thus, for any $\Qtt \in \Qcal$ (corresponding to some $p^* \in [0,1]$) and any stopping time $\tau$, we have
\[
\EE_{\Qtt}[R^{\ML}_\tau] \leq \EE_{\Qtt}[R_\tau^{p^*}] \leq 1,
\]
where the last step invokes the optional stopping theorem for the $\mathrm{Ber}(p^*)^\infty$-NM $(R_t^{p^*})$.

To see that the resulting test has good power, note that  under the alternative,  $p=1/2$ uniquely achieves the above infimum at any even time $t \in 2\NN$, in which case the denominator equals $(1/2)^t$ and the numerator equals one. Thus, $R^{\ML}_t = 2^t$ at even times, and the test $\inf\{t \geq 0:  R^{\ML}_t \geq 1/\alpha\}$ is a valid level $\alpha$ sequential test that stops either at time $\lceil \ln(1/\alpha) / \ln(2) \rceil$ or at time $\lceil \ln(1/\alpha) / \ln(2) \rceil + 1$.

To recap, despite the fact that we cannot find a $\Qcal$-NM, $(R^{\ML}_t)$ is a powerful $\Qcal$-safe e-process against the considered point alternative. This test takes the ratio of the likelihood under the alternative to the maximum likelihood under the null. 
Next, we detail how to handle the composite alternative $\Pcal$ when testing a point null in $\Qcal$.

\subsection{Dealing with the composite alternative using a Jeffreys' mixture}

A standard way of dealing with composite alternatives is to ``mix'' over them by choosing an appropriate mixture distribution. This is sometimes called the Laplace method, pseudo-maximization, or Robbins' method of mixtures~\citep{robbins_boundary_1970,robbins_statistical_1970,robbins_expected_1974,robbins_iterated_1968,robbins_class_1972}. It was first utilized successfully primarily in parametric settings, but has recently re-emerged as a powerful tool for nonparametric inference ~\citep{de_la_pena_pseudo-maximization_2007,howard_uniform_2019,kaufmann2018mixture,waudby2020variance}.  This mixture is sometimes called a ``prior'', but it must be made clear that this is not a prior in the Bayesian sense --- any other distribution also suffices to serve as a mixture in terms of maintaining safety, and no extra assumption is made on the true data when such a ``working prior'' is employed. The choice of mixture does not affect safety, but it does affect computability and power. In the following, we work with a very particular choice of mixture because it yields closed-form expressions, and it yields small constants in regret bounds from the universal coding literature (which has implications for power). 

 Taking independent Jeffreys' priors (with densities $w(\theta) := 1/(\pi \sqrt{\theta(1-\theta)})$) for $p_{1|0}$ and $p_{1|1}$, we obtain the mixture likelihood
\begin{align*}
  \Ptt_{w \times w}(X_1, \ldots, X_t)
  &~:=~
  \int
  \Ptt_{p_{1|0}, p_{1|1}}(X_1, \ldots, X_t)
    w(p_{1|0}) w(p_{1|1}) \dd  (p_{1|0}, p_{1|1})
  \\
  &~=~
  \frac{\Gamma \left(n_{0|0}+\half\right) \Gamma \left(n_{0|1}+\half\right) \Gamma \left(n_{1|0}+\half\right) \Gamma \left(n_{1|1}+\half\right)}{2 \Gamma\left(\half\right)^4 \Gamma (n_{0|0}+n_{1|0}+1) \Gamma (n_{0|1}+n_{1|1}+1)}
  .
\end{align*}
Here we have taken the mixture over all first-order Markov distributions $\Ptt_{{p_{1|0}, p_{1|1}}}$, under which $\{X_1 = 1\}$ has probability $1/2$. 

Thus, for any point null represented by $\mathrm{Ber}(p)^\infty$, we can define the mixture likelihood ratio
\begin{equation}\label{eq:Jeffrey-point-null}
R_t^{\JP,p} := \frac{\Ptt_{w \times w}(X_1,\dots,X_t)}{\prod_{s = 1}^t f_p(X_s)} =  \frac{\text{Jeffreys' mixture over the alternative}}{\text{likelihood under the point null}}.
\end{equation}
Using Fubini's theorem to swap integrals, it is easy to check that $R_t^{\JP,p}$ is a $\mathrm{Ber}(p)^\infty$-NM, and the corresponding sequential test is Wald's usual mixture SLRT~\cite{wald_sequential_1947}. Note that it is the very particular form of this mixture that yields a closed form expression and thus a computationally feasible test. However, we do not use this mixture just for computational reasons; as we detail soon, combining it with the earlier maximum likelihood idea also yields a statistically near-optimal power.

\subsection{Combining Jeffreys' mixture with maximum likelihood}
Using, as in the previous example, that the likelihood under the null is maximised at $p = n_1/t$, where it evaluates to $ \left({n_1}/{t}\right)^{n_1}
    \left({n_0}/{t}\right)^{n_0}$,
we find that \[
R_t := \frac{\text{Jeffreys' mixture over the alternative}}{\text{maximum likelihood under the null}}
\]
reduces to the expression in \eqref{eq:solution}. This is a $\Qcal$-safe $e$-value by combining the arguments used for the safety of $(R^{\JP,p}_t)$ and $(R^{\ML}_t)$: swapping the maximum likelihood with the (unknown) true likelihood, and then employing Fubini's theorem.

We remark that any prior above would have yielded a $\Qcal$-safe $e$-value, and in fact any Beta prior would have yielded one in closed form, but the Jeffreys' mixture above allows us to invoke an appropriate optimal regret bound from the universal coding literature \cite{krichevsky1981} to Markov sources (see \cite{DBLP:journals/tit/TakeuchiKB13} for a discussion of the resulting optimality):
\begin{align}
  R_t
  &\ge
    \notag
  \frac{\tfrac12
    \left(\frac{n_{1|0}}{n_{1|0} + n_{0|0}}\right)^{n_{1|0}}
    \left(\frac{n_{0|0}}{n_{1|0} + n_{0|0}}\right)^{n_{0|0}}
    \left(\frac{n_{1|1}}{n_{1|1} + n_{0|1}}\right)^{n_{1|1}}
    \left(\frac{n_{0|1}}{n_{1|1} + n_{0|1}}\right)^{n_{0|1}}
  }{
    \left(\frac{n_1}{t}\right)^{n_1}
    \left(\frac{n_0}{t}\right)^{n_0}
  } \times e^{-\frac{1}{2} \ln (n_{1|0} + n_{0|0}) - \frac{1}{2} \ln (n_{1|1} + n_{0|1}) - O(1)}\\
  &=
    \label{eq:lbd}
    \frac{\text{maximum likelihood of Markov model}}{\text{maximum likelihood of Bernoulli model}} \times e^{-\ln t - O(1)}.
\end{align}
That is, $(R_t)$ starts gathering evidence against the null if the maximum likelihood for the first-order Markov chain outperforms the maximum likelihood for the Bernoulli model by a factor of order $t$. Note that this is a small hurdle to overcome, as the first term is growing exponentially fast in $t$ when the data are explained better by a Markov model, as argued next.

\subsection{A pathwise theorem on the power of $(R_t)$}

Although our main focus is on first-order Markov alternatives, $(R_t)$ actually has power under much more general alternatives. 
Consider any binary sequence for which the following limits exist:
\begin{subequations}\label{eq:limits}
\begin{equation}
\alpha := \lim_{t \to \infty} \frac{n_{1|1}}{t}, \qquad
\beta := \lim_{t \to \infty} \frac{n_{0|0}}{t}, \qquad
\end{equation}
\begin{equation}
\gamma := \lim_{t \to \infty} \frac{n_{1|0}}{t} = \lim_{t \to \infty} \frac{n_{0|1}}{t}, \qquad
p := \lim_{t \to \infty} \frac{n_1}{t} = \lim_{t \to \infty} \frac{n_{1|1} + n_{0|1}}{t}.
\end{equation}
\end{subequations}
The theorem that follows is a ``pathwise'' result, holding for any sequence where the above limits exist and satisfy an additional weak condition, described later.
These limits exist, for example, under \emph{any} $k$-th order Markov alternative (but may be random unless suitable irreducibility properties hold). Note that the equality of the two expressions for $\gamma$ follows from the fact that $n_{1|0}$ and $n_{0|1}$ can differ by at most one. Similarly, $n_1$ differs from $n_{1|1} + n_{0|1}$ by at most one, which explains why the two expressions for $p$ are equal. Moreover, note that one has the relations
\[
p ~=~ \alpha + \gamma ~=~ 1 - \beta - \gamma ~=~  \frac{p_{1|0}}{p_{1|0} + p_{0|1}},
\]
where we define
\[
p_{1|1} := \frac{\alpha}{p}, \quad
p_{0|1} := \frac{\gamma}{p}, \quad
p_{1|0} := \frac{\gamma}{1-p}, \quad
p_{0|0} := \frac{\beta}{1-p}.
\]
Here we use the convention $0/0:=1$. 
In the first-order Markov case these parameters are simply the transition probabilities.

Within the class of first-order Markov chains, the special case of i.i.d.\ Bernoulli data is characterized by the restriction $p_{1|1} = p_{1|0}$ (and hence $p_{0|1} = p_{0|0}$), or equivalently 
\begin{equation}\label{eq:identifiability}
\alpha (1-p) = \gamma p
\end{equation}
(and hence $\gamma (1-p) = \beta p$). For more general models, such as higher-order Markov chains, the restriction $\alpha (1-p) = \gamma p$ no longer characterizes the i.i.d.\ Bernoulli case. This is illustrated in Example~\ref{ex_2nd_Markov_no_power}, where it is also shown that $(R_t)$ can be powerless against such alternatives (that is, under higher-order Markov alternatives that appear to be Bernoulli when summarized by only first order transition parameters). Nonetheless, as explained in Remark~\ref{R_generic_k_Markov} below, in the suitably ergodic non-Bernoulli $k$-th order Markov case one still has $\alpha (1-p) \ne \gamma p$ \emph{generically}, and therefore achieves power against such alternatives thanks to the following theorem.

\begin{theorem} \label{T:4}
For any data sequence, if the limits $\alpha,\beta,\gamma,p$ in~\eqref{eq:limits} exist, then $\lim_{t \to \infty} \ln R_t/t = r^*$, where
\begin{align}
r^* &:= p \left( \ln\frac{1}{p} - p_{0|1} \ln \frac{1}{p_{1|0}} - p_{1|1} \ln \frac{1}{p_{1|1}} \right)
+ (1-p) \left( \ln \frac{1}{1 - p} - p_{0|0} \ln \frac{1}{p_{0|0}} - p_{1|0} \ln \frac{1}{p_{0|1}} \right) \nonumber \\
&=\gamma \ln \frac{\gamma}{1 - p} + \beta \ln \frac{\beta}{1 - p} + \alpha \ln \frac{\alpha}{p} + \gamma \ln \frac{\gamma}{p} - p \ln p - (1 - p) \ln (1 - p).
\label{eq:210513.1}
\end{align}
Note that the second expression is a difference in entropies between first-order Markovian and Bernoulli sources with the corresponding parameters.
Furthermore, if $\alpha (1-p) \ne \gamma p$ we have $r^* > 0$, that is, $(R_t)$ increases to infinity exponentially fast. 
In fact, $r^* = 0$ if and only if $\alpha(1-p) = \gamma p$. Finally, whenever the exchangeability null is true, $\lim_{t \to \infty} R_t = 0$ almost surely, in addition to $(R_t)$ being $\overline \Qcal$-safe. 
\end{theorem}

Note that whenever the null is true, we necessarily have $r^*=0$. However, $r^*=0$ does not by itself characterize power or consistency; it only characterizes when $(R_t)$ grows exponentially. It is possible to construct certain ``borderline'' examples where $r^*=0$, but $(R_t)$ still increases to infinity (our test is powerful and consistent), albeit at a subexponential rate. We describe and explore such an example in the numerical simulations in Section~\ref{sec:sim-time-varying}.

In the case of a non-Bernoulli first-order Markov alternative, the condition of the theorem is satisfied for almost every realization of the data (that is, with probability one), provided the chain does not have any absorbing states. Indeed, with no absorbing states the Markov chain is recurrent and hence the ergodic theorem applies. This condition is necessary for a power one test, because if (say) $1$ is an absorbing state then there is positive probability of seeing only ones (unless, of course, the Markov chain starts at 0 with probability one). This is indistinguishable from a realization of an i.i.d.\ Ber($1$) sequence.

\begin{remark} \label{R_generic_k_Markov}
    Any $k$-th order Markov chain is characterized by its transition probabilities $p_{1|\boldsymbol s} \in [0,1]$ where $\boldsymbol s$ ranges over all prefixes of length $k$. Note that this determines $p_{0|\boldsymbol s} = 1 - p_{1|\boldsymbol s}$, which we thus do not count as separate parameters. The quantities $\alpha,\beta,\gamma,p$ defined previously, when they are deterministic, are real analytic functions of the transition probabilities. This implies that the condition $\alpha (1-p) \ne \gamma p$ of Theorem~\ref{T:4} holds ``generically'' (in the topological sense, i.e., for an open dense set of parameters). Indeed, a real analytic function is either identically zero, or nonzero on an open dense set, and we know that $\alpha (1-p) = \gamma p$ does not hold identically (i.e., for \emph{all} choices of parameters $p_{1|\boldsymbol s}$). This confirms that $\alpha (1-p) \ne \gamma p$ must hold generically.
\end{remark}

\begin{proof}[Proof of Theorem~\ref{T:4}]
Let $\ell(t)$ denote the logarithm of the ratio between the maximum likelihood for the first-order Markov model and the maximum likelihood for the Bernoulli model. We then have
\begin{align*}
\frac{\ell(t)}{t} &= \frac{n_{1|0}}{t} \ln \frac{n_{1|0}}{n_{1|0} + n_{0|0}}
    + \frac{n_{0|0}}{t} \ln \frac{n_{0|0}}{n_{1|0} + n_{0|0}}
    + \frac{n_{1|1}}{t} \ln \frac{n_{1|1}}{n_{1|1} + n_{0|1}} \\
    &\quad + \frac{n_{0|1}}{t} \ln \frac{n_{0|1}}{n_{1|1} + n_{0|1}}
    - \frac{n_1}{t} \ln \frac{n_1}{t} - \frac{n_0}{t} \ln \frac{n_0}{t}.
\end{align*}
Sending $t \to \infty$ the right-hand side converges to
\[
\gamma \ln \frac{\gamma}{1 - p} + \beta \ln \frac{\beta}{1 - p} + \alpha \ln \frac{\alpha}{p} + \gamma \ln \frac{\gamma}{p} - p \ln p - (1 - p) \ln (1 - p),
\]
which is precisely \eqref{eq:210513.1}. Since $\ell(t) - \ln t + O(1) \le \ln R_t \le \ell(t)$ by \eqref{eq:lbd} we deduce that $t^{-1}\ln R_t$ converges to $r^*$ as claimed.  Re-arranging \eqref{eq:210513.1} we get
\begin{equation} \label{eq_T_4_expr}
r^* = p \left( \ln\frac{1}{p} - \frac{\gamma}{p} \ln \frac{1 - p}{\gamma} - \frac{\alpha}{p} \ln \frac{p}{\alpha} \right)
+ (1-p) \left( \ln \frac{1}{1 - p} - \frac{\beta}{1-p} \ln \frac{1 - p}{\beta} - \frac{\gamma}{1-p} \ln \frac{p}{\gamma} \right).
\end{equation}

 If $\alpha (1-p) = \gamma p$ then one easily verifies that $r^* = 0$. 
Next, suppose $\alpha (1-p) \ne \gamma p$.
Since $\alpha + \gamma = p$, we may bound the first parenthesized expression in \eqref{eq_T_4_expr} using Jensen's inequality,
\[
\ln\frac{1}{p} - \frac{\gamma}{p} \ln \frac{1 - p}{\gamma} - \frac{\alpha}{p} \ln \frac{p}{\alpha}
> \ln\frac{1}{p} - \ln \left( \frac{\gamma}{p} \frac{1 - p}{\gamma} + \frac{\alpha}{p} \frac{p}{\alpha} \right) = 0,
\]
where the inequality is strict because $(1 - p)/{\gamma} \ne p/{\alpha}$. By the same token, the second parenthesized expression in \eqref{eq_T_4_expr} is strictly positive as well. In particular, $r^* > 0$. Finally, if the exchangeability null is true, $(R_t)$ converges to zero almost surely because it is dominated by the likelihood ratio corresponding to two singular distributions on the sequence space.
\end{proof}

\begin{example} \label{ex_2nd_Markov_no_power}
    Here we construct a non-exchangeable 2nd-order Markov chain $(X_t)$ where the above test does not have power one, i.e., $(R_t)$ does not converge to infinity, but instead to zero. This chain satisfies $p = 1/2$ and $\alpha = \gamma$, hence does not yield a counterexample to Theorem~\ref{T:4}  but illustrates the importance of the condition $\alpha (1-p) \ne \gamma p$. 
    To this end, set $p_{1|(0,0)} = p_{1|(0,1)} = 1$ and $p_{1|(1,0)} = p_{1|(1,1)} = 0$. To visualize this 2nd-order Markov chain, note that it alternates between two consecutive zeros and two consecutive ones. For large $t$ we then have $n_{0|0} \approx n_{0|1} \approx 
    n_{1|0} \approx n_{1|1} \approx t/4$ which yields for large $t$ the bound
    \[
        R_{4t} \leq \frac{\Gamma(t + 1)^4}{\Gamma(2t)^2}
            = \frac{(t!)^4}{(2t - 1)!^2}.
    \]
    By Stirling's approximation of factorials, $R_t$ tends to zero as $t$ tends to infinity. 
\end{example}

The above example motivates the following extension, which we only address briefly as a proof-of-concept since the details are conceptually straightforward but tedious.

\subsection{Extension to higher-order Markovian alternatives and context trees}

The aforementioned discussion derived a safe e-process and sequential test for exchangeability of a binary sequence, in particular targeting first-order Markovian alternatives. Neither the binary observations nor the first-order alternatives were particularly critical; these were adopted for clarity of exposition and simplicity of formulae.

Extensions to Markov sources of order $k > 1$ or alphabet sizes $d>2$ are immediate. We may treat each $k$-th order context $x \in \{1,\ldots,d\}^k$ as an independent $d$-ary prediction problem, and by mixing with independent Jeffreys' (which are Dirichlet$(1/2, \ldots, 1/2)$) priors (or equivalently, composing independent Krichevsky-Trofimov estimators), we obtain a computationally attractive e-process with regret bounded by $(d^k(d-1)/2) \ln t + O(1)$. In other words, we get a closed-form e-process $R^{k,d}_t$ --- whose details are tedious, despite being explicit, and thus omitted --- such that
\[
  R^{k,d}_t ~\ge~
  \frac{\text{maximum likelihood of order $k$ Markov model}}{\text{maximum likelihood of Bernoulli model}} \cdot \exp\left(- \frac{d^k (d-1)}{2} \ln t - O(1)\right).
\]
The (near)-optimality of this approach is discussed by~\citet{DBLP:journals/tit/TakeuchiKB13}. The e-process $(R_t)$ from \eqref{eq:solution} can be interpreted as $(R^{1,2}_t)$.

Further computationally attractive extensions include alternatives that consist of Markov sources of varying orders $k=1,2,\ldots$ (see the discussion on the mixture method for unions below). The even more general Context Tree models have the length of the context that should be taken into account depend on that very context \cite{willems1995}.

A similar calculation to the $k=1,d=2$ case done previously shows that $R^{k,d}_t \to \infty$, $\Ptt$--almost surely  for any alternative $\Ptt \in \Pcal_k \backslash \Qcal$, where $\Pcal_k$ is the set of Markovian distributions with order at most $k$. These developments lead naturally to the following subsection.

\subsection{Double-mixtures for a countable sequence of alternatives}
\label{rem:countable-alternatives}

    Let $\Pcal_1, \Pcal_2, \Pcal_3 \dots$ be a countable sequence of alternatives, that may or may not be nested. Suppose for every $k \in \NN$ one can design a safe e-process $(E^k_t)$ for testing $\Qcal$ against $\Pcal_k$ such that it has power one, meaning that for any  $\Ptt \in \Pcal_k \backslash \Qcal$, we have
    \[
    \limsup_{t \to \infty} E^k_t = \infty, \text{ $\Ptt$--almost surely.} 
    \]
    Then, one can design a safe e-process for $\Qcal$ against $\bigcup_{k \in \NN} \Pcal_k$ such that for any $\Ptt \in \bigcup_{k \in \NN} \Pcal^k \backslash \Qcal$, we have
    \[
    \limsup_{t \to \infty}  E_t = \infty, \text{ $\Ptt$--almost surely.}  
    \]
    The proof of the above claim is simple. We can, for example, define the ``double mixture''
    \[
    E_t := \sum_{k=1}^\infty \frac{6}{\pi^2 k^2} E^k_t,
    \]
    which is a countable mixture over the base e-processes (that might have been already mixed using Jeffreys' prior). It is straightforward to check that $(E_t)$ is a safe e-process under $\Qcal$, by invoking monotone convergence and linearity of expectation. To analyze its power, once an alternative $\Ptt$ has been picked, let  $\Pcal_{k^*}$ be the first element of the nested sequence that contains $\Ptt$. Since $\limsup_{t \to \infty}  E^{k^*}_t = \infty$, $\Ptt$--almost surely, the same property holds for $(E^{k^*}_t/k^{*2})$, and thus transfers to $(E_t)$ since e-processes are nonnegative. The computational challenge of calculating $(E_t)$ remains, but this can be reduced by instead calculating the $\Qcal$-safe e-process
    \[
    \widetilde E_t := \sum_{k=1}^t \frac{6}{\pi^2 k^2} E^k_t.
    \]
    At the (finite) time $k^*$, $\widetilde E_t$ begins to include to required term $E_t^{k^*}$, and thus inherits its property of approaching infinity almost surely (consistency). Replacing the sum $\sum_{k=1}^t$ by ~ $\sum_{k=1}^{f(t)}$ for any increasing function $f$ that grows to infinity, possibly with sublinear growth (such as $\ln(\cdot)$), can further save computation without losing the consistency property.

\subsection{\smash{Handling generic alternatives via non-anticipating likelihoods and betting}}\label{sec:betting}
If the alternative to exchangeability is not clearly specified, then a different
approach to the one above may be more suitable, as inspired by the recent work on universal inference by \citet{wasserman2020universal}. It involves a non-anticipating likelihood in the numerator, combined with an MLE in the denominator:
\[
R^{\NA}_t := \frac{\text{non-anticipating likelihood under the alternative}}{\text{maximum likelihood under the null}}.
\]
Here, the numerator is simply given by 
\begin{equation*}
\prod_{s = 1}^t g_s(X_s),
\end{equation*}
where $g_s$ is any ``non-anticipating'' probability mass function, meaning that it is specified before seeing $X_s$, but can be learnt using the first $s-1$ data points; in other words, $(g_t)$ is predictable with respect to $(\Fcal_t)$. One example would be to choose $g_s$ as the (smoothed) maximum likelihood estimator under the alternative using the first $s-1$ samples, but other approaches inspired by machine learning or time series modeling may also be employed. Note that any Bayesian mixture (including ours with Jeffreys' prior) is of this non-anticipating form, with $g_s$ being the associated predictive distribution.

It is easy to prove that $(R^{\NA}_t)$ is a $\Qcal$-safe e-process: each term can be verified to have conditional mean at most one by swapping the denominator for the (unknown) true null likelihood. The major strength of the above approach is that arbitrarily flexible nonparametric or model-free update rules can be used without sacrificing validity, thus opening up the potential for power against loosely specified alternatives or even the discovery of temporal patterns from the observed data. For example, one may employ a complex Bayesian working model that outputs the posterior predictive probability of observing a zero or one at the next step, and this would not violate any of our theoretical guarantees regardless of the choice of priors or working model. Despite such a strong validity guarantee, the current drawback of this approach is that for generic update rules, there may not be an existing regret bound that we may use to convince ourselves of its power. (Of course, such regret bounds would be available for specific update rules and specific alternatives, and the online learning literature is rapidly expanding the scope and types of available regret bounds for individual sequence prediction.)

As a final remark, this non-anticipating likelihood is closely related to the ``predictable-mixture'' approach recently explored by~\cite{waudby2020variance}, and has its roots in \citet[Eq. 10:10]{wald_sequential_1947}. In this vein, it is also closely related to testing hypotheses by betting, as popularized by Shafer and Vovk~\cite{shafer2019language,shafer2019game}; specifically $(g_t)$ can be viewed as a sequence of bets on the following outcome.

\section{Fork-convexity and $\Qcal$-Snell envelopes}\label{sec:fork-snell}

Forgetting for a moment some of the earlier claims made without proof, one of the main questions we seek to answer in this section is:
\begin{quote}
    When is a $\Qcal$-safe $e$-value simply a $\Qcal$-NM or $\Qcal$-NSM in disguise? In other words, is any $\Qcal$-safe $e$-value always improved (or recovered) by some  $\Qcal$-NM or $\Qcal$-NSM?
\end{quote}
Such a question was also asked in the latest preprint on safe testing by~\citet{grunwald_safe_2019}.
The necessity and sufficiency results of~\citet{ramdas2020admissible} imply that the answer in the singleton $\Qcal=\{\Qtt\}$ case is: \emph{always} (via the Doob decomposition of the Snell envelope).
The answer in the composite setting is: \emph{sometimes}. We now qualify the `sometimes' by delving into the rich probabilistic structure underlying safe $e$-values, examining its relationship to convex null sets, a concept called `fork-convexity', and a process that we call a `composite' Snell envelope, known from the mathematical theory of risk measures \cite{MR2276899}.

Most of this section does not depend on our observations being binary, and we  allow the data $(X_t)$ to take values in a more general space $\Xcal$. Some of the technical notions required below, such as local absolute continuity, likelihood ratio (or density) processes, and essential suprema, are reviewed in Appendix~\ref{sec:technical_appendix}.

\subsection{A sequential analog of convexity}

We first review the concept of \emph{fork-convexity}, which can be viewed as a sequential version of convexity.

\begin{definition}
Fix a reference measure $\Rtt$ on the sequence space $\Xcal^\NN$.
\begin{enumerate}
\item A \emph{fork-convex combination} of two locally dominated laws $\Qtt,\Qtt'$ with likelihood ratio processes $(Z_t), (Z_t')$ is another law $\Qtt''$ with likelihood ratio process
\begin{equation}\label{eq_fconv_comb}
Z''_t := \begin{cases}
Z_t, & t \le s \\
h Z_t + (1-h) Z_s \dfrac{Z'_t}{Z'_s}, & t>s 
\end{cases}
\end{equation}
for some $s \in \NN_0$ and some $\Fcal_s$-measurable random variable $h$ in $[0,1]$ with $h=1$ on $\{Z'_s=0\}$. The latter condition ensures that $(Z''_t)$ is well-defined and an $\Rtt$-martingale, as required for a likelihood ratio process.

\item A set $\Qcal$ of probability measures is called \emph{fork-convex} if every fork-convex combination of elements of $\Qcal$ still belongs to $\Qcal$.
\end{enumerate}
\end{definition}

Fork-convexity was first introduced by \v{Z}itkovi\'c \cite{MR1883202}. It is closely related to a concept in the literature on risk measures called \emph{m-stability}, due to Delbaen \cite{MR2276899}. A similar notion called \emph{rectangularity} was introduced by Epstein and Schneider \cite{MR2017864} to describe intertemporal preferences with multiple priors. Rectangularity has then been used extensively in the operations research literature in connection with robust Markov decision processes; see e.g.\ \cite{MR2142033,MR3029483,MR3500621}.

Note that fork-convexity implies convexity. To see this, observe that any (usual) convex combination $a \Qtt + (1-a)\Qtt'$ is also a fork-convex combination; just take $s=0$ and $h=a$ in \eqref{eq_fconv_comb} to get $Z''_t=a Z_t+(1-a)Z'_t$, which is the likelihood ratio process of $\Qtt'':=a \Qtt + (1-a)\Qtt'$. 

A set $\{\Qtt\}$ that consists of a single law is clearly fork-convex. A set $\{\Qtt^1,\Qtt^2\}$ consisting of two distinct laws will not be fork-convex; it is not even convex. However, one can form its ``fork-convex hull''. Here is the general definition.

\begin{definition}[The fork-convex hull and its closure] ~
\begin{enumerate}
    \item The intersection of all fork-convex sets that contain a given set $\Qcal_0$ is called the \emph{fork-convex hull} of $\Qcal_0$. (Note that there is at least one fork-convex set containing $\Qcal_0$, namely the set of \emph{all} laws.)
    \item The \emph{closed} fork-convex hull of $\Qcal$ is the closure of the fork-convex hull of $\Qcal$ with respect to $L^1(\Rtt)$ convergence of the likelihood ratio processes at each fixed time $t \in \NN$, where we recall $\Rtt$ is the assumed reference measure.
\end{enumerate}
\end{definition}

Just as for usual convex hulls, the fork-convex hull of $\Qcal_0$ consists of all finite fork-convex combinations of elements in $\Qcal_0$. Here a \emph{finite fork-convex combination} of some distributions $\Qtt^1,\ldots,\Qtt^n \in \Qcal_0$ is a distribution obtained by iteratively performing \eqref{eq_fconv_comb} a finite number of times on $\Qtt^1,\ldots,\Qtt^n$, on their fork-convex combinations, on \emph{their} fork-convex combinations, and so on. Closed fork-convex hulls play an important role in Theorem~\ref{thm:fork-safe} below.

 To provide some intuition for the definitions as applied to the null $\Qcal$ considered in this paper, one can imagine a more ``algorithmic'' process of producing distributions in the closed fork-convex hull. First pick any $p_1 \in [0,1]$ and observe $X_1 \sim \mathrm{Ber}(p_1)$. Then, after observing $X_1$, pick any $p_2$, and observe $X_2 \sim \mathrm{Ber}(p_2)$. Continue this process indefinitely. Then, the sequence $(p_t)$  is $(\Fcal_t)$-predictable and the resulting binary sequence has a law that is contained in the closed fork-convex hull of $\Qcal$.
 
 It may be instructive to consider another simple example. For a fixed $\mu \in [0,1]$, define $\Qcal^\mu$ as the set of product distributions $\Qtt$ over infinite $[0,1]$-valued sequences $(X_t)$ such that $\EE_\Qtt[X_t | \Fcal_{t-1}] = \EE_\Qtt[X_t] = \mu$, and define $\widetilde \Qcal^\mu$ as the set of  distributions $\Qtt$ (not necessarily of product form) over infinite $[0,1]$-valued sequences $(X_t)$ such that $\EE_\Qtt[X_t | \Fcal_{t-1}] = \mu$.
 Then $\Qcal^\mu$ is not fork-convex if $\mu \in (0,1)$ but $\widetilde \Qcal$ is, and the latter is the closed fork-convex hull of the former. The problem of sequentially estimating $\mu$ in this setup has been recently studied by~\citet{waudby2020variance}.

\subsection{No power against fork-convex hulls}
Consider a null set $\Qcal$ locally dominated by a reference measure $\Rtt$. We now establish the interesting fact that e-processes based on $\Qcal$-NSMs are powerless against any alternative in the closed fork-convex hull of $\Qcal$. We state this formally in Theorem~\ref{thm:fork-safe} below, but the underlying reason is contained in the following lemma.

\begin{lemma}\label{L_fconv_comb_supmg}
If $(L_t)$ is a supermartingale under two laws $\Qtt,\Qtt'$, then $(L_t)$ is also a supermartingale under every fork-convex combination $\Qtt''$ of $\Qtt$ and $\Qtt'$.
\end{lemma}

\begin{proof}
Note that $\Qtt,\Qtt'$ are dominated by $\Rtt := (\Qtt + \Qtt')/2$. 
Fix any $s \in \NN_0$ and $\Fcal_s$-measurable random variable $h$ in $[0,1]$, and let $\Qtt''$ be the fork-convex combination of $\Qtt,\Qtt'$ given in \eqref{eq_fconv_comb}. In compliance with the definition, we restrict $h$ to satisfy $h=1$ on $\{Z'_s = 0\}$. Suppose $(L_t)$ is a supermartingale under $\Qtt$ and $\Qtt'$. Equivalently, $(Z_t L_t)$ and $(Z'_t L_t)$ are supermartingales under $\Rtt$. Thus for $t\in\{1,\ldots,s\}$ we have
\[
\EE_\Rtt[ Z''_t L_t \mid \Fcal_{t-1}] = \EE_\Rtt[ Z_t L_t \mid \Fcal_{t-1}] \le Z_{t-1} L_{t-1} = Z''_{t-1} L_{t-1}.
\]
For $t \ge s+1$ we have  
\[
\EE_\Rtt[ Z''_t L_t \mid \Fcal_{t-1}] = h \EE_\Rtt[ Z_t L_t \mid \Fcal_{t-1} ] + (1-h) Z_s \EE_\Rtt\left[\left. \frac{Z'_t}{Z'_s} L_t \right| \Fcal_{t-1} \right] \le Z''_{t-1} L_{t-1}.
\]
Thus $(Z''_t L_t)$ is an $\Rtt$-supermartingale, or equivalently, $(L_t)$ is a $\Qtt''$-supermartingale.
\end{proof}

The following theorem refers to the \emph{closed} fork-convex hull of $\Qcal$. Recall that this is the closure of the fork-convex hull of $\Qcal$, understood in the sense of $L^1(\Rtt)$ convergence of the likelihood ratio processes at each fixed time $t \in \NN$.

\begin{theorem}\label{thm:fork-safe}
Let $\widetilde\Qcal$ be the closed fork-convex hull of $\Qcal$. Then every $\Qcal$-NSM is in fact a $\widetilde\Qcal$-NSM. Thus a test based on a $\Qcal$-NSM is powerless against $\widetilde\Qcal \setminus \Qcal$.
\end{theorem}

\begin{proof}
The fork-convex hull of $\Qcal$ consists of all finite fork-convex combinations of elements of $\Qcal$. Therefore, thanks to Lemma~\ref{L_fconv_comb_supmg}, every $\Qcal$-NSM remains an NSM under every law $\Qtt$ in the fork-convex hull of $\Qcal$. To extend this to the closure, pick any element $\Qtt \in \widetilde\Qcal$. Then there is a sequence $(\Qtt^n)$ in the fork-convex hull of $\Qcal$ such that $\Qtt^n \to \Qtt$. This means that $Z^n_t \to Z_t$ in $L^1(\Rtt)$ for all $t \in \NN$, where $(Z^n_t)$ and $(Z_t)$ are the likelihood ratio processes of $\Qtt^n$ and $\Qtt$, respectively. By passing to a subsequence, we may assume that $Z^n_t \to Z_t$, $\Rtt$-almost surely, for all $t \in \NN$. Let $(L_t)$ be any $\Qcal$-NSM and hence a $\Qtt^n$-NSM for all $n$. Equivalently, $(Z^n_t L_t)$ is an $\Rtt$-NSM for all $n$. By the $\Rtt$-supermartingale property and the conditional version of Fatou's lemma, we get
\[
\EE_\Rtt[Z_t L_t \mid \Fcal_{t-1}] = \EE_\Rtt\left[\left.\lim_{n \to \infty} Z^n_t L_t \right| \Fcal_{t-1}\right] \le \liminf_{n \to \infty}  \EE_\Rtt[ Z^n_t L_t \mid \Fcal_{t-1}] \le \liminf_{n \to \infty}  Z^n_{t-1}L_{t-1} = Z_{t-1}L_{t-1}.
\]
This completes the proof that every $\Qcal$-NSM is in fact a $\widetilde\Qcal$-NSM.
\end{proof}

The first part of the above theorem asserts that the NSM property is preserved under taking closed fork-convex hulls, but note that this is not true for safe $e$-values in general. Indeed, $(E_t)$ being $\Qcal$-safe implies that it is $\text{conv}(\Qcal)$-safe, but not necessarily $\widetilde \Qcal$-safe.

\subsection{Composite Snell envelopes}

For a single law $\Qtt \in \Qcal$ and an e-process $(E_t)$, the $\Qtt$-Snell envelope is the smallest $\Qtt$-NSM that dominates $(E_t)$. It is natural to ask whether, in contrast to this pointwise construction, one can directly construct a ``composite $\Qcal$-Snell envelope'', i.e., a smallest $\Qcal$-NSM that dominates $(E_t)$. 

It turns out that the ability to define such a $\Qcal$-Snell envelope of an e-process depends heavily on the property of fork-convexity. The following result states that if the null set $\Qcal$ is locally dominated and fork-convex, then a $\Qcal$-Snell envelope of a given e-process $(E_t)$ exists, is safe, and improves upon $(E_t)$.

\begin{theorem}\label{thm:P-Snell}
Let $\Qcal$ be locally dominated and fork-convex. Let $(E_t)$ be a $\Qcal$-safe e-process. Then the process
\[
 L_t := \esssup_{\Qtt \in\Qcal,\, \tau\ge t}\EE_\Qtt[E_\tau\mid\Fcal_t], \quad t \in \NN_0,
\]
where $\tau$ ranges over all finite stopping times, is the smallest $\Qcal$-NSM that dominates $(E_t)$ and satisfies $L_0\le1$. Hence, $(L_t)$ is the $\Qcal$-Snell envelope of $(E_t)$. 
In particular, by the optional stopping theorem, $(L_t)$ is a $\Qcal$-safe e-process. 
\end{theorem}

In this theorem, $L_t$ is defined as the \emph{essential supremum} of the family of random variables $\EE_\Qtt[E_\tau\mid\Fcal_t]$ indexed by $\tau$ and $\Qtt$. This means that $L_t$ is the smallest random variable that almost surely dominates this family. The details of this definition are reviewed in Appendix~\ref{sec:technical_appendix}, along with some key properties. In particular, the proof of the theorem in Appendix~\ref{A:proofs} will make use of Proposition~\ref{P_esssup_closed_max}.

\begin{remark}
We caution the reader that `essential supremum' is used outside of this work with a different meaning: the essential supremum of a (single) random variable $X$ is defined as the smallest constant $c$ such that $X \le c$ almost surely. This notion is different from the one used here, and does not arise in this paper.
\end{remark}

The process $(L_t)$ above is what we call the $\Qcal$-Snell envelope. 
Note that the $\Qcal$-Snell envelope of $(L_t)$ is almost surely equal to $(L_t)$ itself.  In short, the above theorem claims that if $\Qcal$ is fork-convex, then the $\Qcal$-Snell envelope of any $\Qcal$-safe e-process exists and is safe.

To construct a powerful and valid test that dominates a safe e-process $(E_t)$, one might be inherently interested in the \emph{largest} $\Qcal$-NSM $(\overline L_t)$ that dominates $(E_t)$ and satisfies $\overline L_0 \leq 1$. However, we are not aware of a systematic way to obtain such a process. Nevertheless, even the smallest $\Qcal$-NSM that dominates $(E_t)$, namely the $\Qcal$-Snell envelope, still tends to improve its power.

For a given $\Qcal$, can there be more than one process that is considered a $\Qcal$-Snell envelope (of some other process), and amongst these, is there a largest one? In general, the answer is yes for the first question and (typically) no for the second.
 Every $\Qcal$-NSM is its own Snell envelope and there always exist uncountably many $\Qcal$-NSMs, namely the constant and nonnegative decreasing processes starting at one. In particular, the constant process is also a $\Qcal$-NM albeit a powerless one. 
 In fact, there may be uncountably many $\Qcal$-NSMs, with none of these processes dominating the others, and at the same time there may not exist any non-constant $\Qcal$-NMs (that don't use independent external randomization, which involves expanding the filtration). For this paper's choice of $\Qcal$, we later show that every $\Qcal$-NSM is almost surely nonincreasing, and hence the constant process equaling one dominates all $\Qcal$-NSMs, and indeed the only $\Qcal$-NM almost surely equals one.

Taken together, Theorems~\ref{thm:fork-safe} and \ref{thm:P-Snell} lead to the following corollary, which tells us that in certain situations one has to move beyond composite NSMs to achieve powerful tests. We continue to let $\Qcal$ be any locally dominated null set and $\widetilde \Qcal$ its closed fork-convex hull.

\begin{corollary}\label{cor:safe-fork-closure}
Let $(E_t)$ be a $\Qcal$-safe e-process. Then $(E_t)$ is dominated by (or equals) some $\Qcal$-NSM $(L_t)$ with $L_0\le1$ if and only if $(E_t)$ already happens to be $\widetilde\Qcal$-safe (and therefore powerless against $\widetilde\Qcal \setminus \Qcal$).
\end{corollary}

\begin{proof}
   To prove the forward implication, assume $(E_t)$ is dominated by some $\Qcal$-NSM $(L_t)$ with $L_0\le1$. By Theorem~\ref{thm:fork-safe}, $(L_t)$ is in fact a $\widetilde\Qcal$-NSM. It follows that $(E_t)$ is $\widetilde\Qcal$-safe as claimed, because we have $\EE_\Ptt[E_\tau]\le \EE_\Ptt[L_\tau] \le L_0 \le 1$ for every $\Ptt\in\widetilde\Qcal$ and each finite stopping time $\tau$.
    
    To prove the reverse implication, assume that $(E_t)$ is actually $\widetilde\Qcal$-safe. An application of Theorem~\ref{thm:P-Snell} (with $\Qcal$ replaced by $\widetilde\Qcal$) then gives a $\widetilde\Qcal$-NSM $(L_t)$ with $L_0\le1$ that dominates $(E_t)$. This completes the proof of the corollary.
\end{proof}

The above result suggests that we \emph{must} look beyond NSMs for designing sequential tests for exchangeability, and we next show that this fact holds regardless of the class of alternatives considered.

\subsection{The inadequacy of NSMs for testing exchangeability}

We now return to the main focus of this paper, which is binary sequences; thus $\Xcal = \{0,1\}$. In this case, any law $\Ptt$ is locally dominated by the i.i.d.\ Bernoulli(1/2) law $\Rtt := \mathrm{Ber}(1/2)^\infty$ and the likelihood ratio process of $\Ptt$ is
\begin{equation}\label{eq_binary_density_arbitrary}
Z_t = 2^t \prod_{i = 1}^t D_i, \quad \text{where} \quad D_i := q_i(X_1,\ldots,X_{i-1}) \1_{\{X_i=1\}} + (1-q_i(X_1,\ldots,X_{i-1})) \1_{\{X_i=0\}}
\end{equation}
for some functions $q_t \colon \{0,1\}^{t-1} \to [0,1]$ such that, $\Rtt$-almost surely,
\[
q_t(X_1,\ldots,X_{t-1}) = \Qtt(X_t = 1 \mid X_1,\ldots,X_{t-1}).
\]
In particular, taking $q_t = p \in (0,1)$ for all $t$ gives the likelihood ratio process of $\mathrm{Ber}(p)^\infty$ with respect to $\Rtt$. The following theorem shows that any likelihood ratio process of the form \eqref{eq_binary_density_arbitrary} can be approximated by a finite fork-convex combination of likelihood ratio processes corresponding to i.i.d.\ Bernoulli laws. The closed fork-convex hull of this set is therefore very large: it contains \emph{all} laws over binary sequences.

\begin{theorem}\label{thm:every-law-fork-convex}
Every law $\Ptt$ over the space of binary sequences belongs to the closed fork-convex hull of $\Qcal = \{\mathrm{Ber}(p)^\infty \colon p \in (0,1)\}$. Thus, the process $(M_t)$ with $M_t=1$ for all $t$ is the only $\Qcal$-NM, and every $\Qcal$-NSM must be nonincreasing, and hence these never exceed one and always have zero power for any $\alpha \in (0,1)$.
\end{theorem}

\begin{proof}
Fix an arbitrary law $\Ptt$ and let $(Z_t)$ be its likelihood ratio process as in \eqref{eq_binary_density_arbitrary}. We will show by induction that for each $s \in \NN_0$ there exists an element $\Qtt^{(s)}$ in the fork-convex hull of $\Qcal$ whose likelihood ratio process $(Z^{(s)}_t)$ satisfies $Z^{(s)}_t = Z_t$ for all $t \le s$. This is clearly true for $s=0$ since every likelihood ratio process is equal to one at time zero. Suppose it is true for some particular $s$. Consider a law $\Qtt'$ whose likelihood ratio process $(Z'_t)$ satisfies $Z'_{s+1} / Z'_s = 2 D_{s+1}$, where $D_{s+1}$ is defined in \eqref{eq_binary_density_arbitrary}. We can choose $\Qtt'$ from the fork-convex hull of $\Qcal$ by mixing the $2^s$ Bernoulli laws $\mathrm{Ber}(q_{s+1}(y_k))^\infty$ with weights $\1_{\{y_k\}}(X_1,\ldots,X_s)$ at time $s$, where $y_1,\ldots,y_{2^s}$ lists all binary strings of length $s$. Taking the fork-convex combination of $\Qtt^{(s)}$ and $\Qtt'$ at time $s$ with $h=0$ gives a law $\Qtt^{(s+1)}$ in the fork-convex hull of $\Qcal$ whose likelihood ratio process satisfies $Z^{(s+1)}_t = Z^{(s)}_t = Z_t$ for $t \le s$, and $Z^{(s+1)}_{s+1} = 2 Z^{(s)}_s D_{s+1} = Z_{s+1}$ for $t=s+1$. Thus $\Qtt^{(s+1)}$ satisfies the induction assumption with $s+1$ instead of $s$. 

Thus, by induction, for each $s \in \NN_0$ there exists an element $\Qtt^{(s)}$ in the fork-convex hull of $\Qcal$ whose likelihood ratio process $(Z^{(s)}_t)$ satisfies $Z^{(s)}_t = Z_t$ for all $t \le s$, as required. Now, it is clear that for each fixed $t \in \NN$, $Z^{(s)}_t \to Z_t$ in $L^1(\Rtt)$ as $s \to \infty$. This shows that $\Ptt$ belongs to the closed fork-convex hull of $\Qcal$ as claimed in the first part of the theorem.

It remains to prove that every $\Qcal$-NSM $(L_t)$ must be nonincreasing, which implies in particular that the only $\Qcal$-NM is the constant process equal to one. Fix any binary sequence $\omega_0$ and let $\Ptt$ be the law such that $\Ptt(\{\omega_0\}) = 1$. By the first part of the theorem, $\Ptt$ belongs to the closed fork-convex hull of $\Qcal$ and hence, by Theorem~\ref{thm:fork-safe}, $(L_t)$ is a $\Ptt$-NSM. Thus $L_t(\omega_0) = \EE_\Ptt[ L_t ] \le \EE_\Ptt[ L_{t-1} ] = L_{t-1}(\omega_0)$ for every $t \in \NN$. Since $\omega_0$ was arbitrary, this shows that $(L_t)$ must be nonincreasing.
\end{proof}

In other words, not only are $\Qcal$-NSMs inadequate against Markovian alternatives, they are incapable of detecting \emph{any} deviation from exchangeability. We include an overview of the relations we have shown in Figure~\ref{fig:nsm-safety}.

\begin{figure}
\centering
\includegraphics[width=0.25\textwidth]{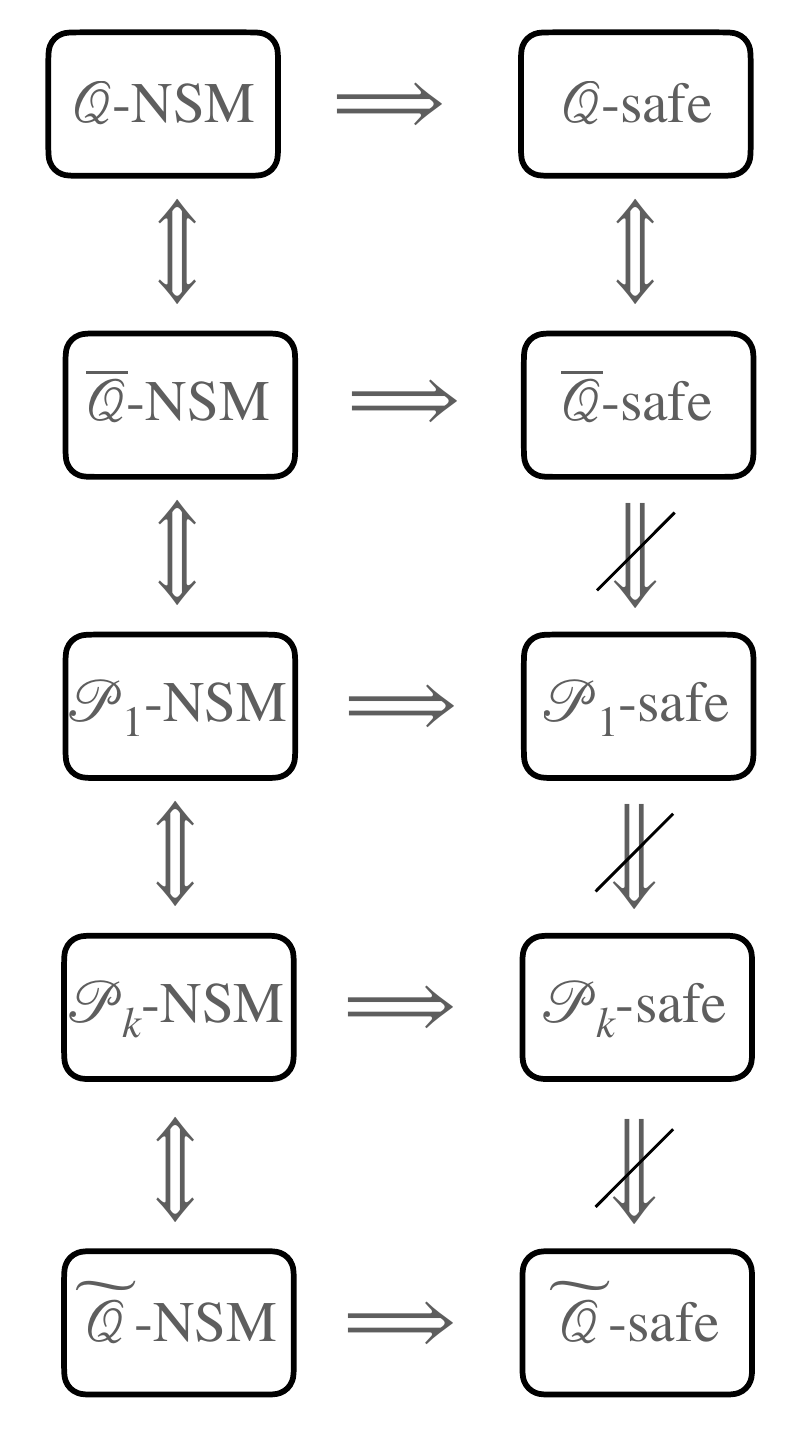}
\caption{A summary of some of the implications related to $\Qcal$-NSMs and $\Qcal$-safety (recall Figure~\ref{fig:nested} for the definitions of the classes $\Qcal$, $\overline \Qcal$, $\Pcal_1$, $\Pcal_k$, and $\widetilde \Qcal$). We would like to design a $\Qcal$-safe e-process that is powerful against $\Pcal_k$.  Theorem~\ref{thm:every-law-fork-convex} proves that a $\Qcal$-NSM is non-viable since it unintentionally results in $\Pcal_k$-safety and thus no power against $\Pcal_k$. The single non-implication sign above opens a door to constructing a non-NSM based $\Qcal$-safe e-process that is consistent against $\Pcal_k$. An example for such an e-process is precisely the construction in~\eqref{eq:solution}.}
\label{fig:nsm-safety}
\end{figure}

\section{Numerical simulations and some extensions}\label{sec:sims}

Recall that it is impossible to detect \emph{all} deviations from exchangeability. For example, if only the first bit of the data is a $\mathrm{Ber}(0.1)$ and the rest is $\mathrm{Ber}(0.2)$, there is simply not enough information available to reject the null. The informal reason is that the deviation from the null is temporary and fleeting, not sustained, and thus easily ascribed to chance. Motivated by this, the primary focus of the paper has been to detect Markovian alternatives. In this section, we will examine the power of our approach against three types of alternatives: (a) the one it was designed primarily for --- a first- (and low-) order Markov alternative, (b) a time-inhomogeneous Markov alternative, and (c) a completely different type of alternative that we were not intending to have power against: a change point alternative.

\begin{remark}
    The last setting should especially not be confused with change point detection, which has different goals, such as minimizing `detection delay' (informally, expected difference between the stopping time and the changepoint, when the procedure stops after the changepoint) subject to a prespecified lower bound on the `average run length' (expected stopping time when there is no changepoint). To be clearer, in order for change point detection algorithms to have a small detection delay, they must pay the price of stopping under the `null' as well, and one typically studies the tradeoffs between those two metrics. However, we will remain still in the realm of sequential testing, where we do not wish to ever stop if the null is true, and the metrics of Type I error and power are rather different from the above. Below, we will focus on measuring evidence using e-processes, where we wish to maximize wealth (or its rate of growth) under the alternative, subject to an upper bound on expected wealth under the null.
\end{remark}

We also reinterpret our test in terms of sequential estimation via confidence sequences in Section~\ref{sec:CS-equiv}, and Section~\ref{sec:calibration} demonstrates the use of calibrators to construct an e-process that approximately tracks the maximum wealth thus far.

\subsection{Sanity check 1: no power against an i.i.d.~Bernoulli sequence} \label{SS:sanity1}
We start with a sanity check, namely that our evidence measure does not report any evidence against Bernoulli sources (which are exchangeable). 
Theorem~\ref{T:4} predicts that when the null is true, $(R_t)$ should decay to zero almost surely, meaning that $(\ln R_t)$ decays to minus infinity. Further, Jensen's inequality implies that at any stopping time $\tau$, $\ln R_\tau$ has nonpositive expectation; hence $\sup_{\Qtt \in \Qcal}  \sup_{\tau}  \EE_\Qtt [\ln R_\tau] \leq 0$, where the second supremum is over all stopping times $\tau$.
Figure~\ref{fig:bern} illustrates that indeed the evidence decays from the start, as expected. The lower bound from \eqref{eq:lbd} predicts that $\ln R_t$ could decay logarithmically in time (but not faster), meaning that $\ln R_t \gtrsim -\ln t$, and that seems particularly tight in the figure below.

\begin{figure}[ht!]
\centering
  \begin{subfigure}{0.45\textwidth}
    \includegraphics[width=\textwidth]{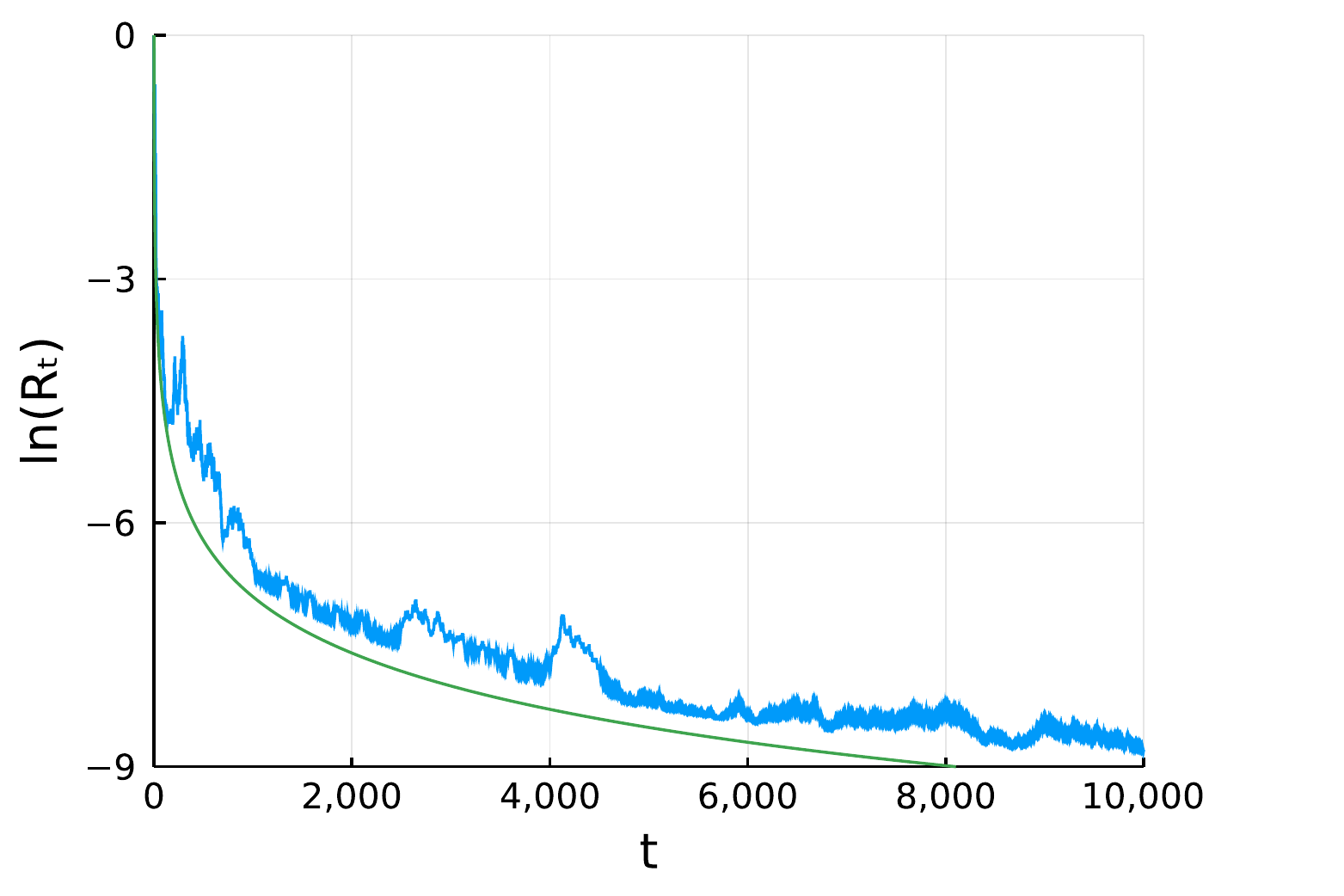}
    \caption{Bernoulli($0.5$)}
    \end{subfigure}
  \begin{subfigure}{0.45\textwidth}
    \includegraphics[width=\textwidth]{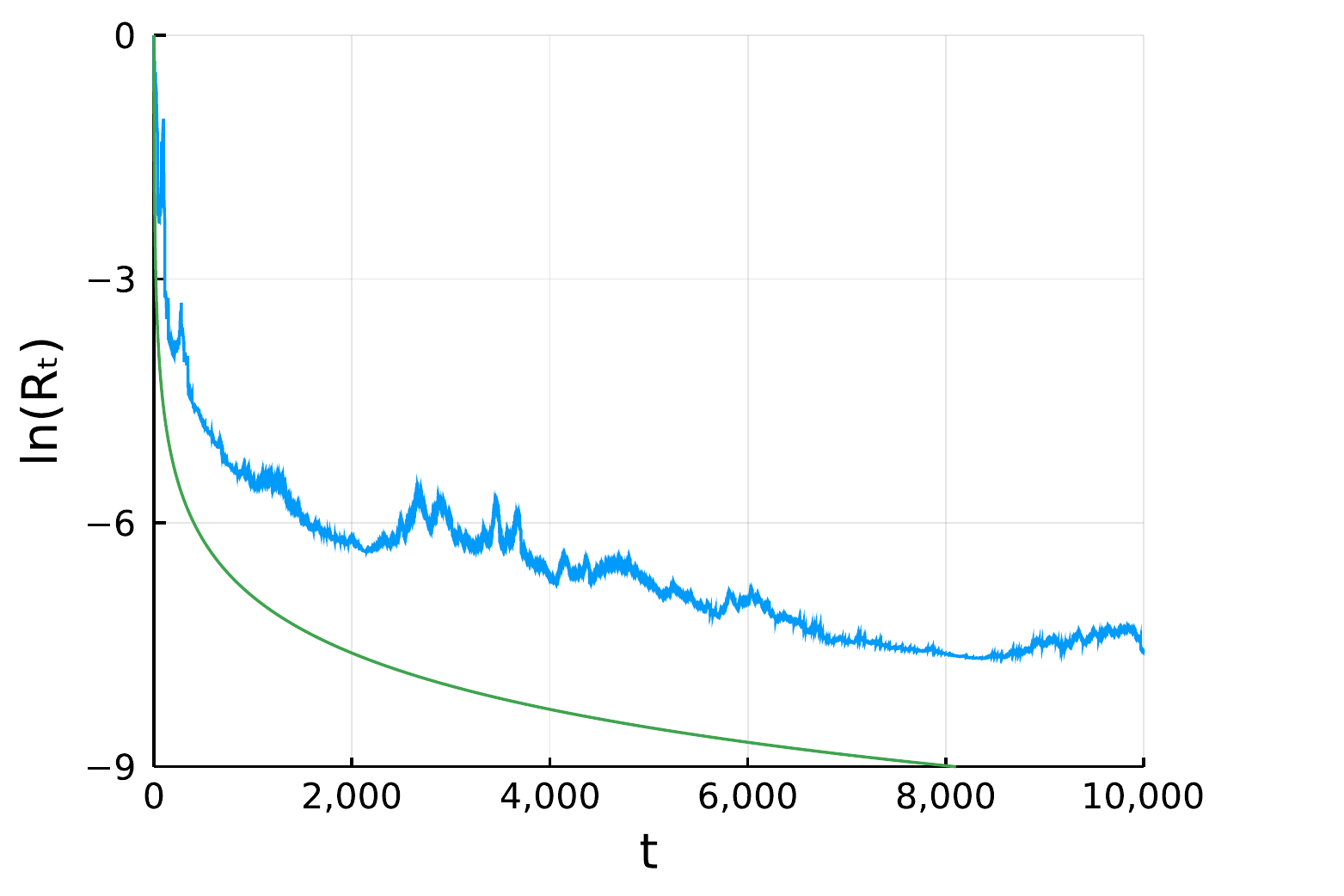}
    \caption{Bernoulli($0.2$)}
  \end{subfigure}
\caption{The blue line shows the evolution of $\ln R_t$ under two Bernoulli sources. The green line is the lower bound $-\ln(t)$ following from \eqref{eq:lbd}.
  The decay is consistent with Theorem~\ref{T:4}, which predicts that $\ln R_t \to - \infty$ almost surely.}\label{fig:bern}
\end{figure}

Having established that the predictions under the null are accurate, we now estimate the power against Markovian alternatives. We first consider a fixed Markov alternative, which the test targets. Then we consider a time-varying Markov alternative, which the test does not explicitly target.

\subsection{Sanity check 2: power against a first-order Markov alternative}
Here we evaluate the evidence measure \eqref{eq:solution} in the well-specified case where data come from a first-order Markov process. Figure~\ref{fig:markov} illustrates that the asymptotics given by Theorem~\ref{T:4} kick in early, as evidenced by the (log)-evidence growing linearly with the sample size.

\begin{figure}[ht!]
\centering
    \begin{subfigure}{0.45\textwidth}
    \includegraphics[width=\textwidth]{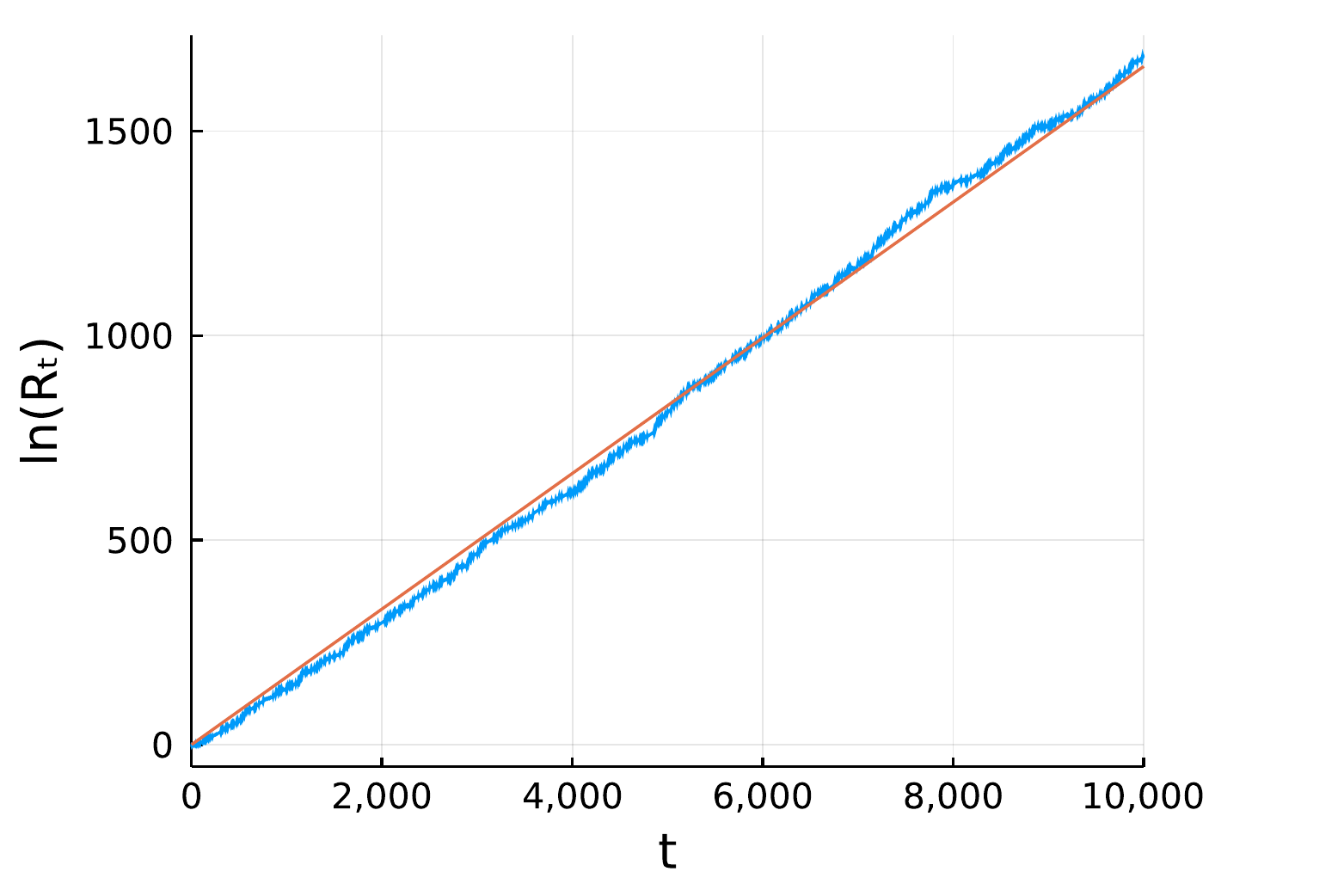}
    \caption{Markov($0.1$, $0.9$)}
  \end{subfigure}
  \begin{subfigure}{0.45\textwidth}
    \includegraphics[width=\textwidth]{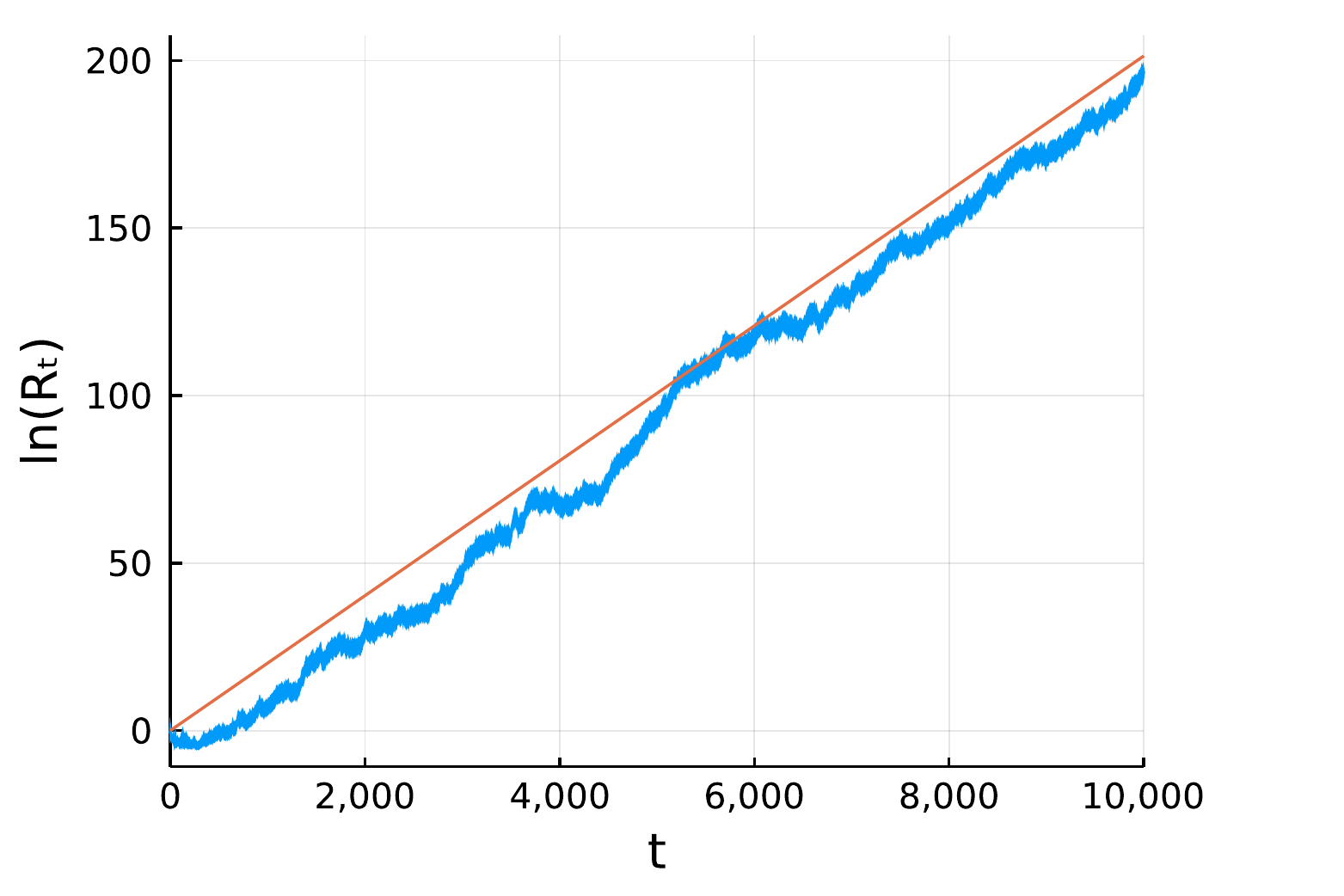}
    \caption{Markov($0.4$, $0.6$)}
  \end{subfigure}
  \caption{
    Blue: the evolution of $\ln R_t$, under a Markov process that is relatively far from Bernoulli (left) and one that is closer (right). Red: the linear approximation with slope $r^*$ from Theorem~\ref{T:4}.
}\label{fig:markov}
\end{figure}

\subsection{Power against a time-varying Markov alternative}\label{sec:sim-time-varying}

Recalling \eqref{eq:limits}, one might conjecture that if a sequence has equal zeroth and first-order frequency limits, then our test will be powerless. Here, we show that this is not necessarily the case.
Namely, we show that it is possible for a sequence to have frequency limits $\lim_{t \to \infty} n_1/t = \lim_{t\to\infty} n_{1|0}/n_0 = \lim_{t\to\infty} n_{1|1}/n_1$ (this happens almost surely under a Bernoulli source) while $(R_t)$ still diverges to infinity, meaning that we  reject the exchangeability null at any confidence level. We will consider a time-varying Markov source with symmetric transition probabilities $p_{0|0}(t) = p_{1|1}(t) = 1/2 + \delta_t$, for some decreasing sequence $\delta_t$ tending to $0$. To make the example interesting, we need two properties of the ``repetition bias'' $\delta_t$:
\begin{itemize}
\item
  We aim to pick $\delta_t$ small enough so that the above limit frequencies are ${1}/{2}$. In particular, this requires $\lim_{t \to \infty}  \sum_{s=1}^t \delta_s/t = 0$. For example, $\delta_t := {1}/{t^\alpha}$ for $\alpha >0$ would do for this purpose.
\item
  We aim to pick $\delta_t$ large enough so that $\ln R_t \to \infty$. To this end, we can use the Law of the Iterated Logarithm to conclude that after $t$ rounds $n_0$ and $n_1$ are both $t/2 \pm O(\sqrt{t \ln \ln t})$ while $n_{0|0}$ and $n_{1|1}$ are both $ t/4 +  \sum_{s=1}^t \delta_s/2 \pm O(\sqrt{t \ln \ln t})$. Let us write $\ell(t)$ for the likelihood ratio of the Markov maximum likelihood and the Bernoulli maximum likelihood, as we did in the proof of Theorem~\ref{T:4}. Abbreviating the binary entropy to $h(p) := -p \ln{p} -(1-p)\ln(1-p)$ and using that $h(1/2+\delta) = \ln 2 - 2 \delta^2 + O(\delta^3)$ reveals that
    \begin{equation}\label{eq:appxl}
    \ell(t)
    ~=~
      t \left(
      h\left(\frac{n_1}{t}\right)
      - \frac{n_0}{t} h\left(\frac{n_{0|0}}{n_0}\right)
      - \frac{n_1}{t} h\left(\frac{n_{1|1}}{n_1}\right)
    \right)
    ~\approx~
      \frac{2}{t} \left(
      \sum_{s=1}^t \delta_s \pm O(\sqrt{t \ln \ln t})
    \right)^2
    .
  \end{equation}
  To have the lower bound $\ln R_t \ge \ell(t) - \ln t - O(1)$ from \eqref{eq:lbd} diverge, we need to ensure
  \[
    2 \left(\frac{1}{\sqrt{t}} \sum_{s=1}^t \delta_s - O\left(\sqrt{\ln \ln t}\right) \right)^2
    - \ln t \to \infty.
  \]
  This reveals that the threshold for divergence is around $\sum_{s=1}^t \delta_s = \sqrt{t/(2 \ln t)}$ (this being forced by the $\ln t$ term on the right, not by the $\sqrt{\ln \ln t}$ term inside the square). To investigate the behaviour ever so slightly above this threshold, we ended up picking $\delta_t := F(t+1)- F(t)$ for $F(t) := \min \{t/2, \sqrt{t} \ln(1+t)\}$, where the $\min$ ensures that $\delta_t \le {1}/{2}$ from $t=1$ without affecting the order of growth. The results are displayed and discussed in Figure~\ref{fig:pincer}.
\end{itemize}

\begin{figure}[h!]
\centering
  \begin{subfigure}[t]{0.45\textwidth}
    \includegraphics[width=\textwidth]{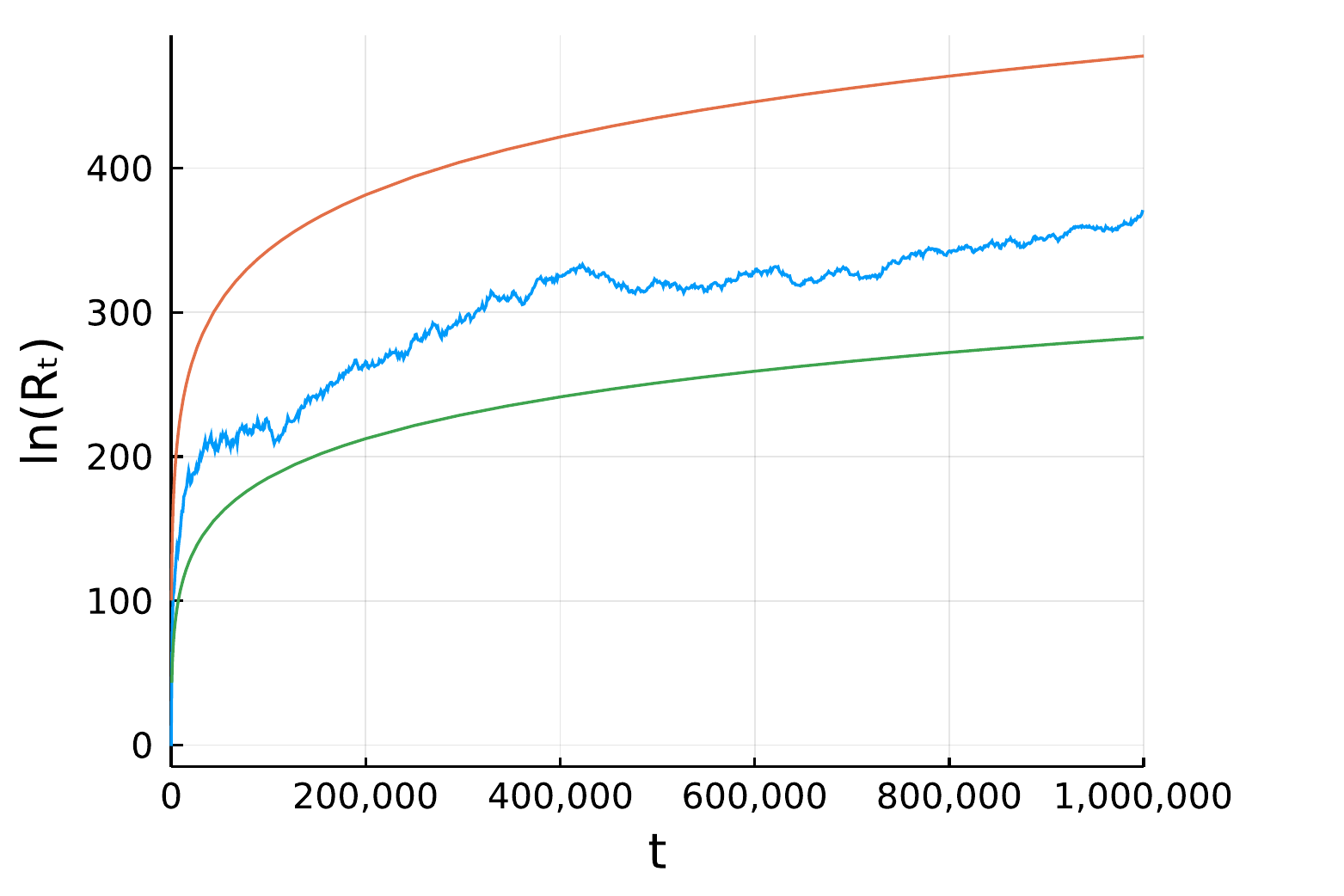}
    \caption{Markov(${1}/{2} -\delta_t$, ${1}/{2} + \delta_t$). This source is ``sticky'', in that it slightly favours repeating the previous outcome. The evidence against the null diverges.
      Red is the approximation \eqref{eq:appxl} with the constant in $O(\cdot)$ taken to be $1$, while green is \eqref{eq:appxl} with the constant in $O(\cdot)$ as $-1$ and $\ln(t)$ subtracted.}
  \end{subfigure}
  \quad
  \begin{subfigure}[t]{0.45\textwidth}
    \includegraphics[width=\textwidth]{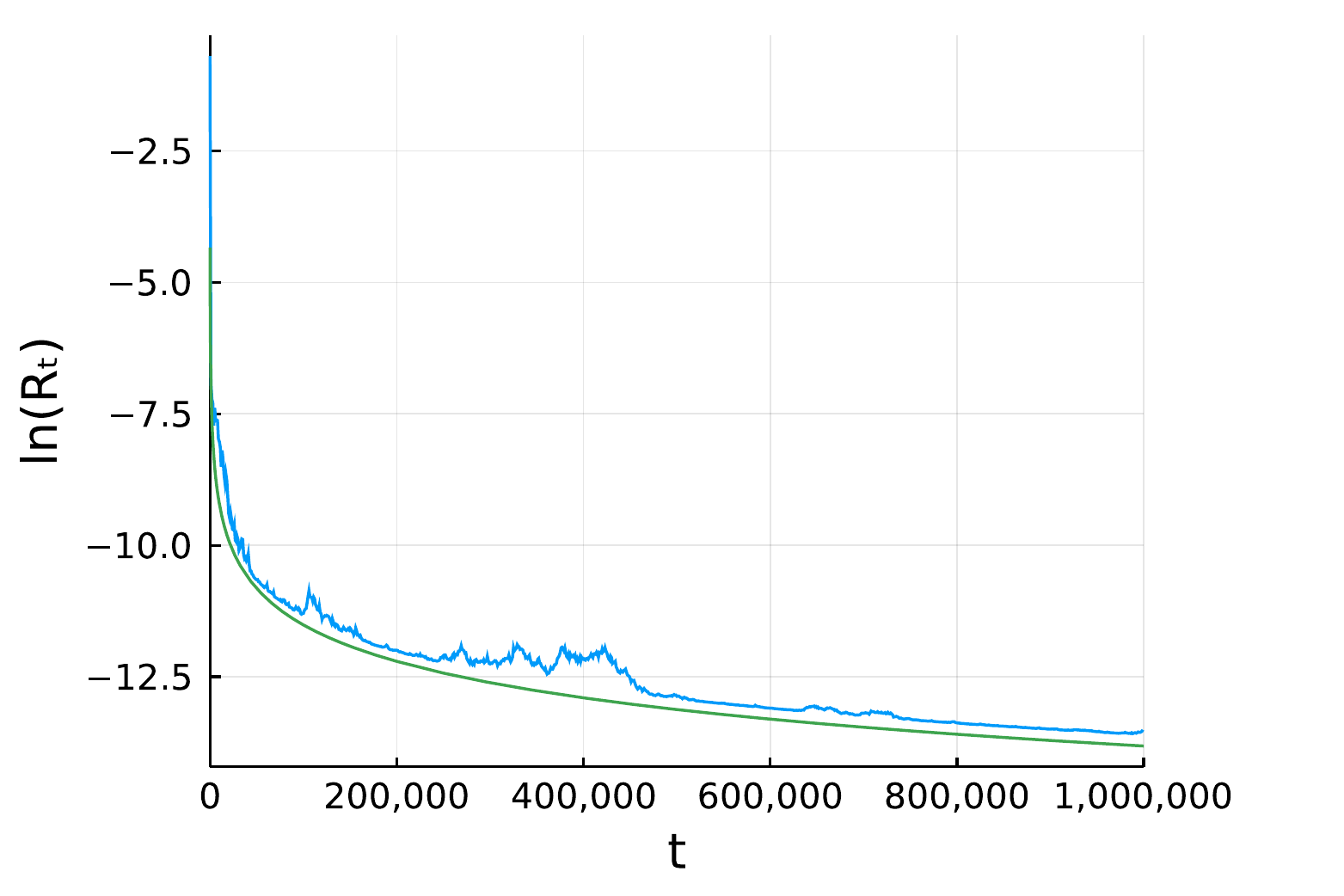}
    \caption{Bernoulli(${1}/{2}+\delta_t$) = Markov(${1}/{2}+\delta_t$, ${1}/{2}+\delta_t$). This source is one-sided, in that it slightly favours ones. No evidence against the null accumulates. The green line is the lower bound $-\ln(t)$ following from \eqref{eq:lbd}.}
  \end{subfigure}
  \caption{
    The evolution of our evidence $(\ln R_t)$ for a time-varying Markov process. On the left the transition probabilities are $p_{1|1}(t) = p_{0|0}(t) = 1/2+\delta_t$ for $\delta_t := F(t+1) - F(t)$ with $F(t) := \min \{t/2, \sqrt{t} \ln(1+t)\}$. On the right, we perform a sanity check with $p_{1|1}(t) = p_{1|0}(t) = 1/2+\delta_t$ instead, which renders this a time-varying Bernoulli process. As predicted, on the left the log-evidence diverges, growing at a logarithmic rate (which is far slower than the linear rate of Figure~\ref{fig:markov}). The right picture confirms that this is due to the Markovian aspect of the source, and not due to the changing probabilities.
  }\label{fig:pincer}
\end{figure}

Despite the generality and pathwise nature of Theorem~\ref{T:4}, this experiment points towards the possibility of strengthening it even further. Indeed, Theorem~\ref{T:4} has a very natural sufficient condition for consistency, but this experiment shows that even when this condition does not hold (meaning that the limits~\eqref{eq:limits} exist and the condition~\eqref{eq:identifiability} fails), then our safe e-process can still be consistent in certain subtle settings. These subtleties were apparent to us when designing this example, but we are not yet sure how to formalize these observations in a general fashion, and thus leave this to future work.

Another such example that demonstrates the unexpected power of our approach is examined next.


\subsection{Power against a change point alternative}\label{sec:changepoint}
We consider a somewhat counterintuitive example to show that a first-order Markov alternative to exchangeability is perhaps more powerful than one may believe at first sight. Consider a length $2n$ sequence of coin flips sampled from $\mathrm{Ber}(p)^{n} \mathrm{Ber}(q)^{n}$ for some $p \neq q$. To match the setup of this example with our initial problem set up, one could potentially extend this to an infinite sequence in an arbitrary way, for example just continuing as $\mathrm{Ber}(q)$ after time $2n$. 

This sequence is clearly not exchangeable. It is, however, not clear whether our proposed first-order Markov alternative would detect (much) evidence against the null, as the sequence is not Markov, but is more like a change point alternative. Detecting evidence is not a given; the outcomes are in fact independent (albeit not identically distributed). Hence there is no first-order dependency structure for the Markov model to exploit. And on top of that, there seems to be only one problematic time-point, precisely half-way through the sequence. So even if the Markov model somehow exploited this, how could it gain an amount of evidence growing with the length $n$ of the sequence?

We now show that the above arguments are all misguided, and that the process $(R_t)$ from \eqref{eq:solution} gains an amount of evidence against the exchangeable null that grows exponentially with $t$, between time $n$ and $2n$. The evolution of $(R_t)$ on a typical run of this process is shown in Figure~\ref{fig:Rt}.

Initially, $(R_t)$ loses steam and tends towards zero at a rate $1/t$ before time $n$ since the null is true and there is a price to pay for the Jeffreys' mixture over the alternative. To calibrate what to expect after the change point, think of $n$ as being relatively large so that we can reason about empirical frequencies of zeros and ones with more ease. Let us compute the maximum likelihood parameters for typical sequences with frequency (tending to) $p$ in the first half and $q$ in the second half. For the Bernoulli model, we find $\hat p = (p+q)/{2}$. For the first-order Markov model we find that
\[
  \hat p_{1|1} ~=~
  \frac{
    p^2+q^2
  }{
    p + q
  }
  \qquad
  \text{and}
  \qquad
  \hat p_{1|0} ~=~
  \frac{
    (1-p) p + (1-q) q
  }{
    (1-p)
    + (1-q)
  }.
\]
The main observation here is that the best Markov model is i.i.d.\ if $\hat p_{1|1} = \hat p_{1|0} = \hat p$, which occurs if and only if $p=q$. The fact that an exploitable first-order Markov dependency structure arises can perhaps be best observed in the extreme case $p=0$ and $q=1$. As this comparison does not really depend on $n$, we find that for all other parameter settings with $p \neq q$, the Markov model will gain overall evidence exponentially growing with $t$ between time $n$ and $2n$. (Technically, the exponential growth does not start immediately at time $n+1$, but it does so eventually.) However, as $t$ grows even further --- say beyond $t=n^2$ or $t=2^n$ --- $(R_t)$ will decrease once more towards zero. This is because the sequence eventually is dominated by i.i.d.\ $\mathrm{Ber}(q)$ coin flips, and the MLE under the null explains the data very well.

\begin{figure}[t!]
  \centering
  \includegraphics[width=0.45\textwidth]{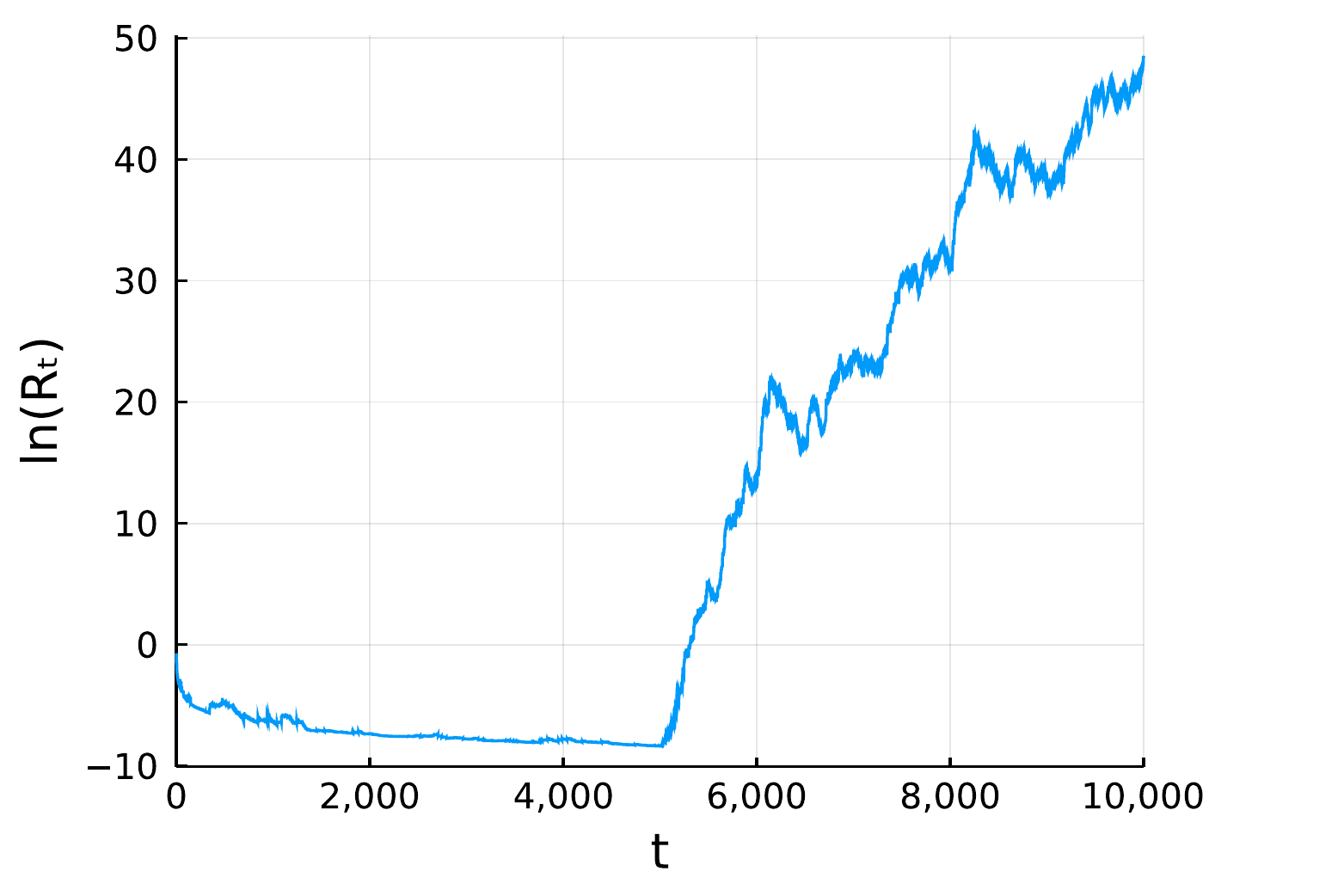}
  \caption{The process $(\ln R_t)$ on a sequence sampled from $\mathrm{Ber}(.1)^{5000}\mathrm{Ber}(.4)^{5000}$. On the first half, we see $(R_t)$ decays as $1/t$, which is due to the overhead of Jeffreys' mixture for the Markov model over the maximum likelihood Bernoulli parameter. After the change point, we see $(R_t)$ increasing fast on the exponential scale. Recalling~\eqref{eq:pvalue} for those more familiar with the p-value scale, the corresponding anytime p-process dips below $10^{-20}$ towards the end.
  }\label{fig:Rt}
\end{figure}

Thus, for this example, we do not get a power one test, nor should we expect a single change point away from an i.i.d.\ model to yield power one for a test designed to be powerful against Markovian alternatives. 
However, if the initial pattern repeats itself after $2n$ steps, meaning that we keep alternating between $\mathrm{Ber}(p)$ and $\mathrm{Ber}(q)$ models, then $(R_t)$ does have power one, and this is interesting because $(R_t)$ is designed for first-order Markov alternatives, but it is consistent against these $2n$-th order Markov alternatives.

In fact, one can argue that it is information theoretically impossible to design \emph{any} power one test, including tests that are tuned to detect a single change point. To see why, think of very small $n$, like $n=2$, to make the reason intuitively transparent. How can a test possibly have enough evidence with just two coin flips before the change point, to know with probability one a change actually did occur? Naturally, the larger the time $n$ of the change point, the higher the power could be of any such test (as it is for our test also), but no test can possibly have power one since there is always some small probability (vanishing with $n$) that the distribution of the first $n$ coin flips looks similar to the post-change distribution.

Nevertheless, this simple example illustrates the point that our proposed e-process $(R_t)$ for evidence against exchangeability is actually powerful even in scenarios that are not (close to) Markov.

\paragraph{Deriving other e-processes targeted towards change point alternatives.}

    For readers explicitly interested in powerful tests to detect change point alternatives in the setting of this paper, we briefly describe a powerful test (albeit not a power one test, as already explained above). Essentially, one can combine the ideas in Remark~\ref{rem:countable-alternatives} with those in Section~\ref{sec:betting}. We let $\Pcal_k$ denote the alternative in which the change point is hypothesized to occur at time $n=2^k$, though other increasing functions of $k$ may also suffice. We will define an e-process $(E^k_t)$ for each $k$ and then use a countable mixture over $k$ as the final e-process.
    
    Now, we describe an e-process for a fixed $k$.
    Define $(g^{-}_s)_{s=1}^n$ to be a smoothed non-anticipating maximum likelihood estimator, calculated using data from time $1$ to $s-1$. The smoothing step is simple: add a single fake observation worth half a heads (or half a tails) to the counts when determining the MLE. The smoothing leads to a slight regularization that can be viewed as the maximum-a-posteriori estimate using a $\text{Beta}(1/2,1/2)$ prior, analogous to Krichevsky-Trofimov betting~\cite{krichevsky1981}.
    Similarly, define $(g^{+}_s)_{s=n+1}^\infty$ to be the same smoothed non-anticipating maximum likelihood estimator, but calculated using data from time $n+1$ to $s-1$. In both cases, the smoothing also leads to a well-defined function $g^-_1$ and $g^+_{n+1}$, which are effectively treated as a $\mathrm{Ber}(1/2)$ model. 
    Finally, define the e-process $(E^k_t)$ as the ratio of $\prod_{s=1}^{n \wedge t} g^-_s(X_s) \prod_{s = (n+1)}^t g^+_s(X_s)$ to the maximum likelihood under the i.i.d.\ null. In other words, the denominator is identical to one of $R_t$, but the numerator has changed because the targeted alternative is now different. 
    
    Recalling Section~\ref{sec:betting} (and the final section of~\citet{wasserman2020universal}), it is easy to see that $(E^k_t)$ is a $\Qcal$-safe e-process. If a change point occurs at time $n^*$ (and let $k^* := \lfloor \ln_2 n^* \rfloor$), the e-processes $E^{k^*}_t$ and $E^{k^*+1}_t$ will grow exponentially between time $n^*$ and $2n^*$. Even with the countable weighting of Remark~\ref{rem:countable-alternatives}, their exponential growth washes out the inverse polynomial weights, to yield a powerful e-process $(E_t)$.
    
    Naturally, many permutations and combinations of these ideas (see e.g.\ the literature on switching experts \cite{ehmms}) can be used to derive a variety of tests against different kinds of alternatives. See Section~\ref{sec:vovk-conformal} and~\cite{vovk2021retrain} for an approach based on conformal prediction. We leave further exploration of these variants to future work.

\subsection{Power on a real-world data set}\label{sec:real.data}
We conclude the empirical evaluation with a jovial yet practical question: if we track the days on which it rains in London, do we get an exchangeable sequence? Three of us having lived there deemed this hypothesis unlikely yet not implausible, but a Markovian alternative certainly seems a-priori more plausible. To answer our question, we obtained daily areal rainfall data for North London from \url{https://data.london.gov.uk/dataset/daily-areal-rainfall}, spanning the 4404 days between 01/01/2007 to 21/01/2019. We binarised the data by comparing the daily total millimeters to zero, resulting in 59\% rainy days overall. As is clear from Figure~\ref{fig:rain}, the null hypothesis can be safely rejected.

\begin{figure}[ht!]
  \centering
  \includegraphics[width=.45\textwidth]{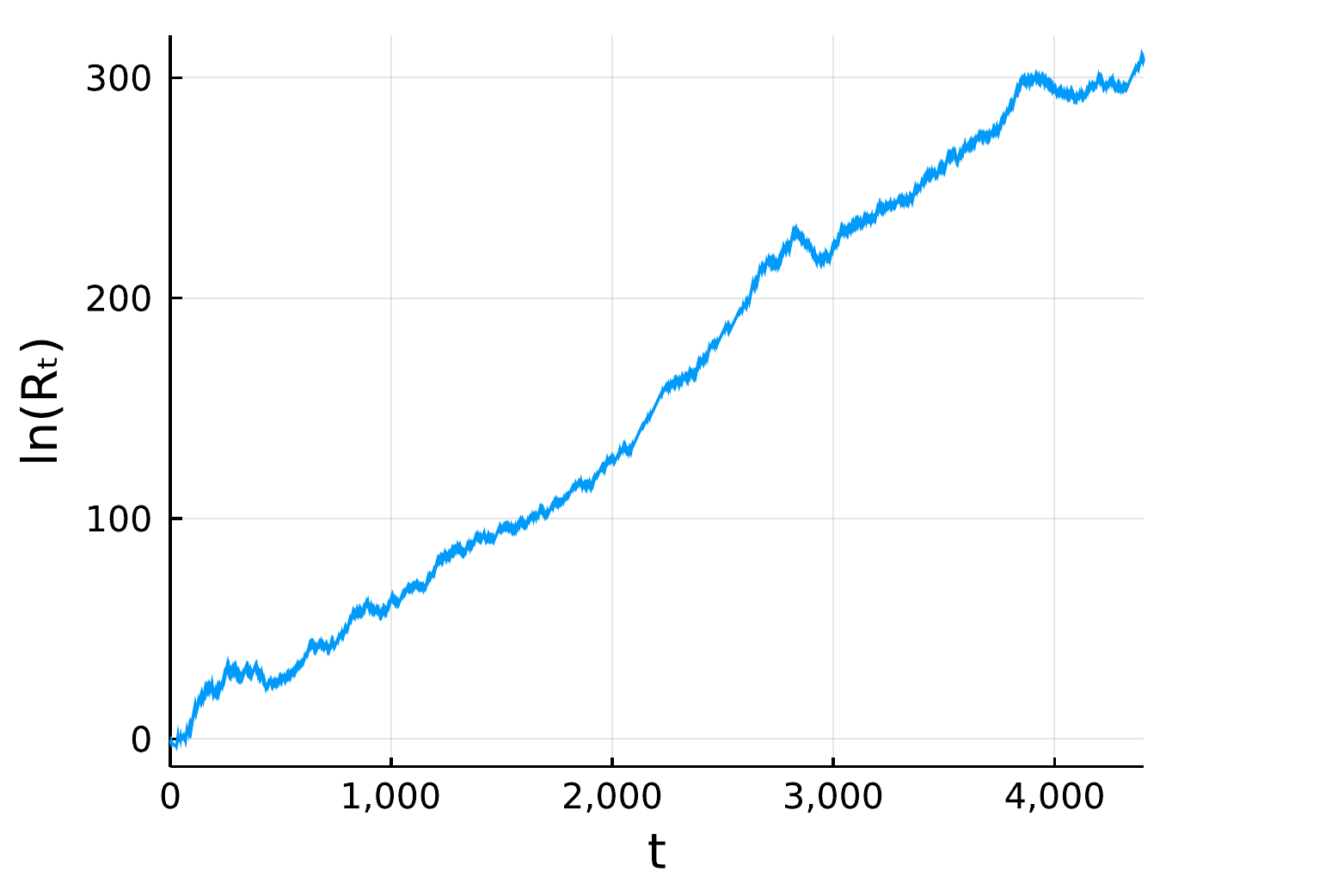}
  \caption{The process $(\ln R_t)$ accumulating evidence against the hypothesis that daily rain in North London is an exchangeable process.}
  \label{fig:rain}
\end{figure}

\subsection{An equivalent viewpoint based on confidence sequences}\label{sec:CS-equiv}

It is possible to view our sequential test --- rejecting the null if $(R_t)$ ever exceeds $1/\alpha$ --- in terms of sequential estimation using confidence sequences~\cite{darling_confidence_1967,howard_uniform_2019,ramdas2020admissible}. 
Denoting $\Qtt_p := \mathrm{Ber}(p)^\infty$, a $(1-\alpha)$-confidence sequence for $p$ is a sequence of confidence sets $(C_t)_{t \geq 1}$ such that $\Qtt_p(\exists t\geq 1: p \notin C_t) \leq \alpha$. Recalling~\eqref{eq:Jeffrey-point-null}, define 
\[
C_t := \left\{ q :  R^{\JP,q}_t < \frac{1}{\alpha}\right\}.
\]
To see that $(C_t)$ is a confidence sequence for $p$, note that $p \notin C_t$ if and only if $R^{\JP,p}_t \geq 1/\alpha$, but Ville's inequality yields that $\Qtt_p(\exists t \geq 1: R^{\JP,p}_t \geq 1/\alpha)$ since $R^{\JP,p}_t$ is a $\Qtt_p$-NM.

Based on the above argument, Theorems~\ref{thm:valid} and \ref{T:4} together imply the following:
\begin{itemize}
\item Rejecting the null as soon as $(C_t)$ becomes empty is identical to rejecting if $(R_t)$ ever exceeds $1/\alpha$:
\begin{equation}\label{eq:CS-stop}
\inf\{t \geq 1: C_t = \emptyset\} = \inf\left\{t \geq 1: R_t \geq \frac{1}{\alpha}\right\}.
\end{equation}
This is because $R_t = \inf_{q \in [0,1]} R^{\JP,q}_t$, making the decision rules identical.
\item If the data $X_1,X_2,\dots$ are exchangeable, then with probability at least $1-\alpha$, $C_t$ is always nonempty. Indeed, it is possible to argue that $C_t$  will shrink to a single point with probability $\geq 1-\alpha$. Indeed, a single draw from $\Qtt \in \overline \Qcal$ is simply an i.i.d.~$\mathrm{Ber}(p)^\infty$ sequence for some $p$; and for any $q \neq p$,  $R^{\JP,q}_s$ will eventually exceed $1/\alpha$ because Jeffrey's mixture  likelihood (a mixture over all first-order Markov chains, including the special case that recovers an i.i.d.~$\mathrm{Ber}(p)$ sequence) will explain the data much better than the $\mathrm{Ber}(q)$ likelihood. This can be formalized via the regret bound~\eqref{eq:lbd} but this simple argument is omitted for brevity.
\item Since $C_t$ contains the aforementioned $p$ under the null (simultaneously for all $t$ with probability at least $1-\alpha$), we can infer that $\bigcap_{t=1}^\infty C_t \equiv 
 \left\{ q : \sup_{t \geq 1} R^{\JP,q}_t <1/\alpha\right\}$ is also nonempty with probability at least $1-\alpha$. This means that the decision rule $\inf\{t \geq 1: \bigcap_{s \leq t}C_s = \emptyset\}$ yields a valid level $\alpha$ sequential test that could potentially halt sooner than \eqref{eq:CS-stop}. (The advantage is only expected to be slight in practice.)
\item For any sequence of observations where the limits in~\eqref{eq:limits} exist but~\eqref{eq:identifiability} fails to hold, $C_t$ will become empty at some point (deterministically). As explained around Theorem~\ref{T:4}, the aforementioned failure of~\eqref{eq:identifiability} happens with probability one under first-order Markovian alternatives, for example.
\end{itemize}
Figure~\ref{fig:CIs} illustrates these points using the running intersection $\bigcap_{s \leq t} C_s$ on synthetic and real-world data.

\begin{figure}[ht!]
  \centering
  \begin{subfigure}[t]{0.47\textwidth}
    \includegraphics[width=\textwidth]{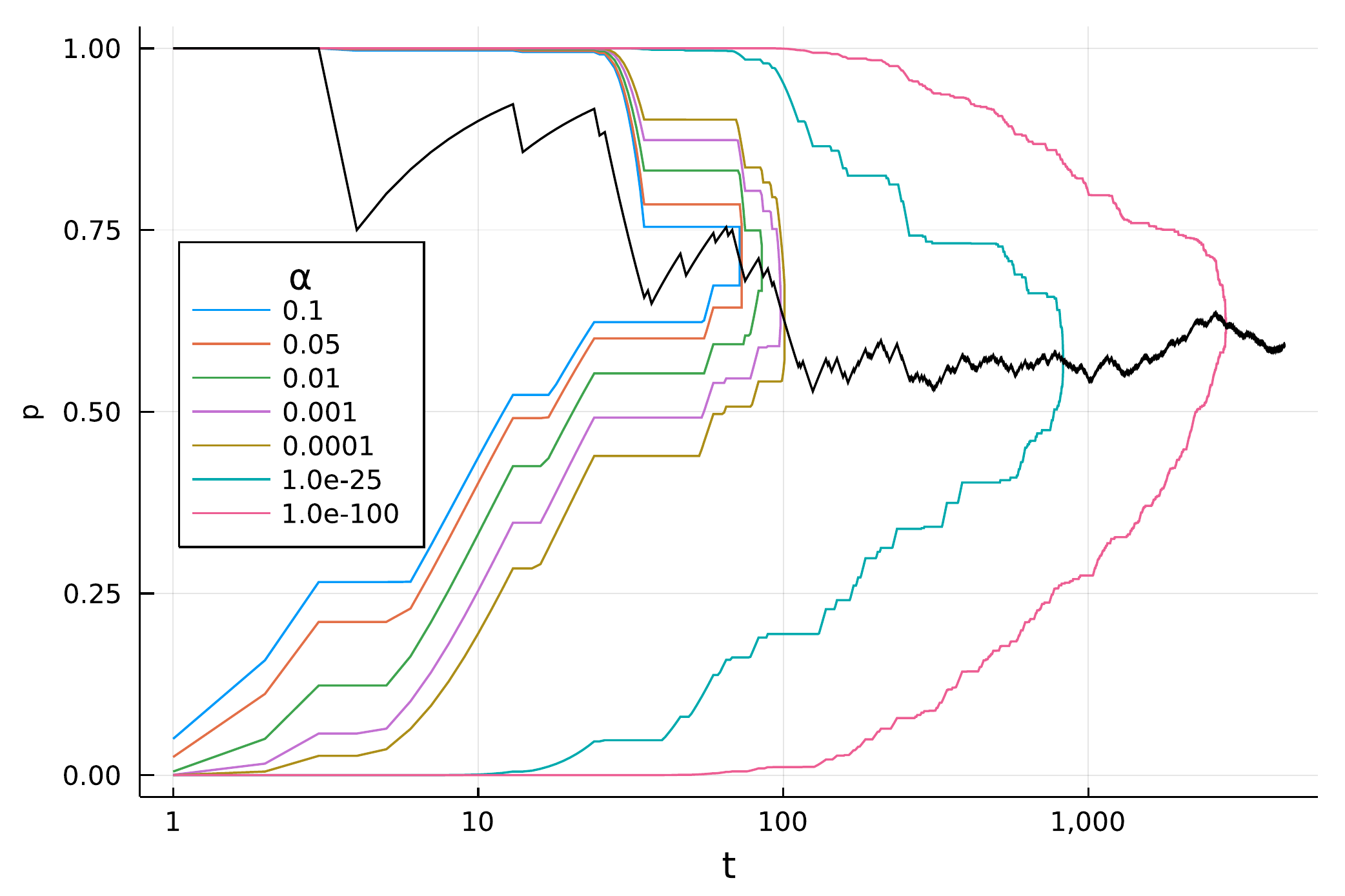}
    \caption{The rain dataset from Section~\ref{sec:real.data}. At all $\alpha$ considered, the confidence interval becomes empty, meaning that we reject the null.}
  \end{subfigure}
  \quad
  \begin{subfigure}[t]{0.47\textwidth}
    \includegraphics[width=\textwidth]{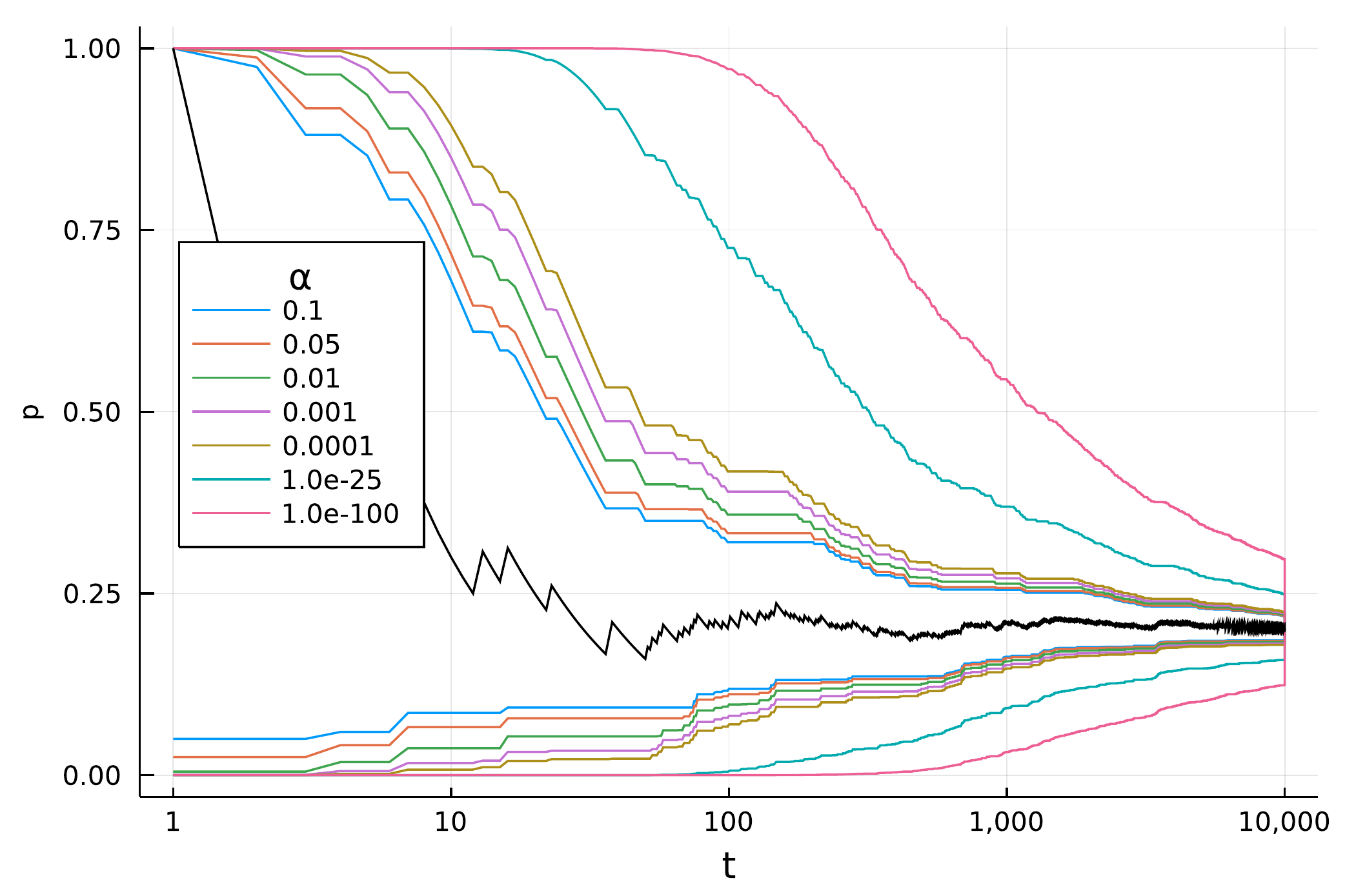}
    \caption{$\mathrm{Ber}(0.2)$ (as in Section~\ref{SS:sanity1}). Here the null is true, and indeed all confidence intervals stay non-empty.}
  \end{subfigure}
  \caption{The confidence sequence (anytime-valid confidence interval) $\bigcap_{s \leq t} C_s$ for $p$ is depicted by its lower and upper boundaries as a function of $t$, at various coverage probabilities $\alpha \in (0,1)$. The running average of the data is displayed in black.
  As predicted, the confidence sequence becomes empty when the null is false (left), while it stays non-empty when the null is true (right).
  }
  \label{fig:CIs}
\end{figure}

\subsection{Calibrated p-processes and adjusted e-processes for not losing capital}\label{sec:calibration}

While $(R_t)$ is a $\Qcal$-safe e-process, $(\max_{s\leq t} R_s)$ is not. In other words, we are only allowed to measure our performance based on the wealth accumulated thus far and not the highest wealth that we reached at some point in the process. The same is not true for p-processes: $(1/R_t)$ is a $\Qcal$-valid p-process, and so is $(1/\max_{s\leq t} R_s)$, the latter being the running infimum of the former. In game-theoretic terminology, the gambler can decide to stop playing the game (betting against the null) according to any stopping rule $\tau$, but once they have stopped, only the final wealth $R_\tau$ of the gambler matters, and a nearly bankrupt gambler cannot point to their past wealth as a measure of their proficiency. This subtle point particularly manifests itself in the above change point example, because with a single change point, $(R_t)$ rises to some amount (above $10^{20}$ in the figure) and then will shrink back to zero, so if we happen to stop too late, then $R_\tau$ could provide only meagre evidence even though it was once astronomically large.

So how can we get around this worrisome issue? We take inspiration from~\citet{shafer_test_2011} and use ``calibrated p-processes'' as our e-processes. (As a matter of terminology, our use of calibration here can be seen as an $E \to P \to E$ process, but if we skip the middle step entirely, the $E \to E$ direct method has been called ``adjustment'' by~\citet{DAWID2011157,dawid2011probability,koolen2014buy}. We will present it from both angles below to tie some loose ends in the literature together.)

Define $\Pval_t := 1/\max_{s\leq t} R_s$, so that $(\Pval_t)$ is a $\Qcal$-valid p-process that satisfies~\eqref{eq:pvalue}. Let $f$ be a calibrator~\cite{shafer_test_2011,ramdas2020admissible}, which is a nonincreasing function $f$ such that $\int_0^1 f(u) \dd u = 1$. Then $(f(\Pval_t))$ is a $\Qcal$-safe e-process. It is not hard to check that $f(\Pval_t) \leq R_t$, so there is some price to pay for being able to take the best possible wealth into account. One possible choice for $f$ is given by
\[
f(u):= \frac{1-u+u \ln u}{u (-\ln u)^2} ~ ;
\]
also see \citet[Eq. (2)]{vovk2019values}.

In order to do things more directly, let $F$ be an adjuster~\cite{shafer_test_2011,dawid2011probability}, which is an increasing function $F$ such that $\int_1^\infty F(y) y^{-2} \dd y = 1$. Defining $A_t := F(\max_{s \leq t} R_s)$ yields that $(A_t)$ is a $\Qcal$-safe e-process, and indeed as before  $A_t \leq R_t$. One possible choice for $F$ is given by
\[
F(y) := \frac{y^2 \ln 2}{(1+y) (\ln(1+y))^2}.
\]
Thus, even if $(R_t)$  rises sharply and then decreases to zero eventually, $(A_t)$ does not since it is nondecreasing. In fact, using the $F$ given above in our example with a single change point, and noting that $F(y) \asymp y / (\ln y)^2$ for large $y$, we see that $A_\infty \approx 10^{17}$ even though $R_\infty = 0$. Of course, if $R_t \to \infty$ then so does $F(R_t)$, meaning that it does not lose the consistency property against Markovian alternatives.

Thus, at a (squared) logarithmic price to the overall capital, one can be protected against future losses, and for this reason we recommend using $A_t = F(R_t)$ as an e-process if we are uncertain about how close our alternative might be to the idealized Markovian case studied here.

\section{Summary and discussion}\label{sec:discussion}

\subsection{From convex hulls to fork-convex hulls}

The celebrated theorem of de Finetti --- for which many proofs exist including based on elementary arguments~\cite{kirsch2018elementary} --- states that all exchangeable binary sequences are mixtures of i.i.d.\ sequences (this was generalized much beyond the binary setting by Hewitt and Savage~\cite{hewitt1955symmetric}). In fact, for any exchangeable sequence, the empirical measure $P_t := (1/t) \sum_{s=1}^t \delta_{X_s}$ converges in distribution to a measure $\mu$ supported on $[0,1]$, and this is the so-called ``de Finetti mixing measure'' alluded to in the previous sentence. The crux of the matter is that the convex hull of all i.i.d.\ distributions over infinite binary sequences is precisely the set of exchangeable distributions. Since the convex hull preserves properties like safety (of e-processes) and validity (of sequential tests and p-processes), one can develop tests for the i.i.d.\ setting, which is itself a nontrivial composite null, and invoke de Finetti's theorem to extend the result to the exchangeable setting. 

In this paper, we go several steps further: we prove that the set of Markovian sequences lies inside the ``fork-convex hull'' of all exchangeable (or i.i.d.)\ sequences. In fact, Theorem~\ref{thm:every-law-fork-convex} shows that the closed fork-convex hull is so large that \emph{every} law over binary sequences is contained in it!

The fork-convex hull of a set of distributions can be informally thought of as a ``predictable mixture of these distributions'' (here, we borrow the terminology of predictable mixtures from~\citet{waudby2020variance}, where it was applied to processes). Speaking informally in language inspired from game-theoretic probability, if Reality wanted to draw an infinite sequence from the fork-convex hull of $\Qcal$, it would pick $\Qtt^1$ and draw (only) $X_1$ from it, then based on that the outcome, Reality would pick $\Qtt^2$ and draw (only) $X_2$ from its conditional distribution given $X_1$, and so forth ad infinitum, each time using the observed data to pick the next (conditional) distribution as it wishes. 

Fork-convexity was a central object in this paper, and one of its primary roles was in dealing a fatal blow to tests based on constructing nonnegative supermartingales, as summarized next.

\subsection{The powerlessness of test supermartingales, and a powerful e-process}

Theorem~\ref{thm:fork-safe} shows that the nonnegative supermartingale (NSM) property is preserved not just by taking the convex hull of a set of distributions, but also when taking the (much larger) fork-convex hull.
Using composite Snell envelopes, Corollary~\ref{cor:safe-fork-closure} shows that if a safe test for $\Qcal$ is upper bounded by a $\Qcal$-NSM, then it must have already been safe for the fork-convex hull of $\Qcal$ (and thus powerless). 
Together, these results show that any  NSM under exchangeable distributions is also an  NSM under Markovian distributions, and in fact it is an NSM under \emph{every} distribution over binary sequences. In other words, test statistics that are NSMs (or ``test supermartingales'') are nonincreasing sequences, rendering them powerless to distinguish non-exchangeable distributions from exchangeable ones. 

We get around the above hurdles by designing a safe e-process $(R_t)$ in~\eqref{eq:solution} that is upper bounded by some nonnegative martingale for every exchangeable distribution, despite not being an NSM itself. (The idea of designing a process that is upper bounded by an NSM, despite not being one itself, also appears elsewhere in the sequential testing literature~\cite{howard_exponential_2018}, but for different reasons.) Our safe e-process uses the method of mixtures with Jeffreys' prior to handle the composite alternative, along with the maximum likelihood under the null, to ultimately yield a computationally efficient closed-form e-process. This e-process not only has the desired safety properties at arbitrary stopping times (Theorem~\ref{thm:valid}), but we prove that it has power one against any first-order Markovian alternative, and also a generic dense set of higher-order Markovian alternatives (Theorem~\ref{T:4}). For first-order alternatives, the e-process also grows at an optimal rate, as implied by a regret bound~\eqref{eq:lbd} borrowed from the universal coding literature.

Section~\ref{sec:jeffreys+ML} also describes how to derive other e-processes that work for higher-order Markovian alternatives, and finally also for even more general, loosely specified alternatives by combining the method of predictable mixtures~\cite{waudby2020variance} and betting~\cite{shafer2019language}, along with universal inference~\cite{wasserman2020universal}.

\subsection{Vovk's approach based on conformal prediction}\label{sec:vovk-conformal}

An alternate approach towards testing exchangeability was recently expounded by~\citet{vovk_testing_randomness_2019}, which is based on conformal prediction. It replaces the canonical filtration $(\Fcal_t)$ by a poorer filtration $(\Gcal_t)$ formed by conformal p-values. (We refer to their paper for the technical details.)
Vovk then produces a sequence of independent p-variables under the null, which are converted to e-processes by appropriate calibration, which are in turn combined to form a martingale with respect to $(\Gcal_t)$.  This is particularly interesting because, despite the only martingales with respect to $(\Fcal_t)$ being constants, Vovk is able to identify nontrivial martingales with respect to an appropriately impoverished filtration $(\Gcal_t)$.

Our approaches based on Jeffreys' mixture and the nonanticipating likelihood (or predictable mixture) can be seen as providing two alternatives to Vovk's methodology. 
Vovk's algorithm seems particularly powerful for change point alternatives, making them most similar to the extensions discussed in Section~\ref{sec:changepoint}, while our paper focuses more on Markov alternatives. Further, since Vovk's martingale property only holds with respect to $(\Gcal_t)$, so are the set of stopping times for which his method has the safety property. In other words, the algorithm is not allowed to look at the raw data and decide when to stop; it must only look at the sequence of p-variables. In contrast, our method is safe with respect to a much larger class of stopping times, indeed \emph{all possible} stopping times (with respect to $(\Fcal_t)$, the canonical filtration). However, in the end, the details of both works appear to be very different, and the conceptual principles by which the methods are derived also differ significantly.


A final, alternate approach to this problem could utilize \emph{reverse} martingales and \emph{exchangeable} filtrations.
To elaborate, the exchangeable
filtration is the reverse filtration
$(\mathcal E_t)_{t=0}^\infty$ 
where $\mathcal E_0 := \sigma(\{X_1, X_2, \dots\})$, 
and for all $t \geq 1$,
$\mathcal E_t$ denotes the $\sigma$-algebra
generated by all 
real-valued Borel-measurable functions
$f(X_1, X_2, \dots)$ which are permutation-symmetric in their 
first $t$ arguments, so that $\mathcal E_0 \supseteq \mathcal E_1 \supseteq \mathcal E_2 \ldots$.
It is known that if the data are exchangeable, then the empirical
measure $P_t := (1/t) \sum_{s = 1}^t \delta_{X_s}$
forms a measure-valued reverse martingale
with respect to the exchangeable filtration, in the sense
that $(\int g \dd P_t)$, is a reverse martingale
for any bounded and Borel-measurable function $g$
\citep{kallenberg2006}. In fact, the converse of this statement
also holds true if the sequence $(X_t)$
is stationary \citep{bladt2019}. This fact has recently been exploited to develop confidence sequences and sequential tests in other contexts~\cite{manole2021sequential}. We hope to explore in more detail whether this approach can lead to powerful tests for exchangeability in the future.

\subsection{A game-theoretic protocol and reinterpretation of Theorem~\ref{T:4}}\label{sec:game-theoretic}

Despite our work being grounded in measure theory, it has some implications for game theoretic probability, as developed by Shafer and Vovk, amongst others. Due to lack of space, we presume the reader has some general familiarity with their setup and refer to their latest book~\cite{shafer2019game} for details. We recall from~\citet{shafer2019language} the basic protocol of testing by betting:
\begin{quotation}
    \noindent A skeptic begins with initial wealth of one, $\mathcal K_0:=1$.\\
    \textit{for $t=1,2,\dots$}\\
    --- A forecaster announces a distribution $Q_t$ for $X_t$.\\
    --- A skeptic announces a bet $S_t: \{0,1\} \to \RR^+$ such that $\EE_{Q_t}[S_t(X_t)] \leq 1$.\\
    --- Reality announces the outcome $x_t$.\\ 
    --- The skeptic's wealth gets updated as $\mathcal K_t := \mathcal K_{t-1}\cdot S_t(x_t)$.
\end{quotation}
At any point in time, the skeptic's wealth $\mathcal K_t$ acts as a measure of evidence against the forecaster. Instead of announcing a single distribution $Q_t$, the forecaster could announce a set of distributions $\Qcal_t$, in which case the bet $S_t$ must satisfy the constraint that $\sup_{Q_t \in \Qcal_t} \EE_{Q_t}[S_t(X_t)] \leq 1$. While this protocol is very general, the setup for the current paper is better understood using the following protocol.
\begin{quote}
A forecaster F announces a set of distributions $\overline{\Qcal}$ for the entire sequence.\\
A skeptic S wishes to challenge F, and observes that $\overline{\Qcal}=\text{conv}(\bigcup_{m \in \mathcal M}\Qcal^m)$, for some set $\mathcal M$, where $\Qcal^m$ is itself a set of distributions.\\
S requests F to play a different game for each $m \in \mathcal M$, starting each with one dollar, $\mathcal K^m_0:=1$.\\
F accepts the request, as long as the Skeptic's net wealth is measured by its worst performance across all games, that is, $\mathcal K_t := \inf_{m \in \mathcal M} \mathcal K^m_t$.\\
\textit{for $t=1,2,\dots$}\\
    --- For each $m \in \mathcal M$, Skeptic places a bet $S^m_t: \{0,1\} \to \RR^+$ such that $\smash{\sup_{Q^m_t \in \Qcal^m_t} \EE_{Q^m_t}[S^m_t(X_t)] \leq 1}$,
    where $Q_t^m(X_t) := Q^m(X_t|X^{t-1})$ for each $Q^m \in \Qcal^m$ and $\Qcal^m_t:=\{Q^m_t : Q^m \in \Qcal^m\}$.\\
    --- Reality announces the (common across games) outcome $x_t$.\\ 
    --- The skeptic's wealth in the $m$-th game gets updated as $\mathcal K^m_t := \mathcal K^m_{t-1}\cdot S^m_t(x_t)$.
\end{quote}
 
In this paper, of course, $\mathcal M := [0,1]$, $m:=p$, and $\Qcal^m := \{\mathrm{Ber}(p)^\infty\}$ is a singleton. The skeptic's wealth $\mathcal K_t := \inf_{m \in \mathcal M} \mathcal K^m_t$ would be the game-theoretic instantiation of an e-process (while in the earlier setup, it would have been a supermartingale). Even though we do not employ the added generality, it is worth noting that both the above protocols are just special cases of the following protocol that affords the forecaster more flexibility.

\begin{quote}
S and F agree to play a set of games indexed by $m \in \mathcal M$, starting each with one dollar, $\mathcal K^m_0:=1$. The Skeptic's net wealth is measured by $\mathcal K_t := \inf_{m \in \mathcal M} \mathcal K^m_t$.\\
\textit{for $t=1,2,\dots$}\\
  --- For each $m \in \mathcal M$, Forecaster announces a set of distributions $\Qcal^m_t$ on $\mathcal{X}_t$.\\
    --- For each $m \in \mathcal M$, Skeptic places a bet $S^m_t: \mathcal{X}_t \to \RR^+$ such that $\smash{\sup_{Q^m_t \in \Qcal^m_t} \EE_{Q^m_t}[S^m_t(X_t)] \leq 1}$.\\
    --- Reality announces the (common across games) outcome $x_t \in \mathcal{X}_t$.\\ 
    --- The skeptic's wealth in the $m$-th game gets updated as $\mathcal K^m_t := \mathcal K^m_{t-1}\cdot S^m_t(x_t)$.
  \end{quote}

The reader may also refer to Chapter 10 of~\citet{shafer2019game} for related ideas that are presented somewhat differently there.

We now reinterpret our main theorem in a game-theoretic language, inspired by the original results of Ville's PhD thesis~\cite{ville_etude_1939}, which gave a gambling interpretation to measure-zero sets. In particular, Ville proved that for any event of measure zero (say, under $\mathrm{Ber}(1/2)^\infty$), one can design a betting strategy (a nonnegative martingale) whose wealth increases to infinity whenever that event occurs; see Proposition~8.14 in \citet{shafer2005probability}. Inspired by this result, and the fact that Theorem~\ref{T:4} is a pathwise statement that holds true for every path, we can restate it in the spirit of Ville's work. To this end, define the set 
\[
A:=\{\omega \in \{0,1\}^\infty: \text{ either the limits in~\eqref{eq:limits} do not exist, or they exist and condition~\eqref{eq:identifiability} holds}\}.
\]
Note that the set $A$ satisfies $\inf_{\Qtt \in \overline \Qcal} \Qtt(A) = 1$, meaning that for an exchangeable sequence, the aforementioned limits (almost) always hold and the corresponding condition is (almost) always satisfied. In other words, its complement satisfies $\sup_{\Qtt \in \overline \Qcal}\Qtt(A^c) = 0$. Then, Theorem~\ref{T:4} states that our process $(R_t)$ increases to infinity whenever $A^c$ occurs. The safety property of $(R_t)$ corresponds to it being a valid betting strategy under the second (and third) game-theoretic protocol presented above. Thus, informally, $(R_t)$ could be seen as an explicitly constructed ``witness'' to Ville's theorem in the above context. 

Deriving a fully general, `robust' version of Ville's theorem that holds for composite nulls $\Qcal$ (and yet recovers Ville's original theorem for singleton nulls) appears to be an important open direction.

\subsection*{Acknowledgments}
AR acknowledges NSF DMS grant 1916320. The authors also thank the reviewers for insightful feedback that led to significant improvements to the paper.


{
\hypersetup{linkcolor=red}
\bibliography{shafer}
\bibliographystyle{plainnat}
}

\appendix

\section{Additional technical concepts and definitions} \label{sec:technical_appendix}

\subsection{Reference measures and local absolute continuity}
\label{sec:ref_measure}

Consider a probability space with a filtration $(\Fcal_t)_{t \in \NN_0}$. Let $\Rtt$ be a particular probability measure on $\Fcal_\infty$; we think of $\Rtt$ as a \emph{reference measure}.
We now explain the concept of local domination and how it allows us to unambiguously define conditional expectations.

\begin{itemize}
    \item If $\Ptt$ is a probability measure on $\Fcal_\infty$ and $\tau$ is a stopping time, we write $\Ptt|_\tau$ for the restriction of $\Ptt$ to $\Fcal_\tau$. (This is simply the probability measure on $\Fcal_\tau$ defined by $\Ptt|_\tau(A)=\Ptt(A)$, $A\in\Fcal_\tau$. Think of this as the `coarsening' of $\Ptt$ that only operates on events observable up to time $\tau$.)
    
    \item $\Ptt$ is called \emph{locally dominated by $\Rtt$} (or \emph{locally absolutely continuous with respect to $\Rtt$}), if $\Ptt|_t\ll \Rtt|_t$ for all $t\in\NN$. We write this $\Ptt\ll_\text{loc}\Rtt$. More explicitly, this means that
    \[
    \text{$\Rtt(A)=0$}\quad\Rightarrow\quad \Ptt(A)=0, \quad \text{ for any $A\in\Fcal_t$ and $t\in\NN$}.
    \]
    Local absolute continuity does \emph{not} imply that $\Ptt\ll \Rtt$. However, it does imply that $\Ptt|_\tau\ll \Rtt|_\tau$ for any finite (but possibly unbounded) stopping time $\tau$. Indeed, if $A\in\Fcal_\tau$ and $\Rtt(A)=0$, then $A\cap\{\tau\le t\}\in\Fcal_t$ for all $t$, and hence $\Ptt(A)=\lim_{t\to\infty} \Ptt(A\cap\{\tau\le t\})=0$.
    
    \item A set $\Pbb$ of probability measures on $\Fcal_\infty$ is called locally dominated by $\Rtt$ if every element of $\Pbb$ is locally dominated by $\Rtt$.
    
    \item Any $\Ptt\ll_\text{loc}\Rtt$ has an associated \emph{likelihood ratio process} (often also called \emph{density process}), namely the $\Rtt$-martingale $(Z_t)$ given by $Z_t:=\dd\Ptt|_t/\dd\Rtt|_t$. Being a nonnegative martingale, once $Z_t$ reaches zero it stays there. Thus with the convention $0/0:=1$, the ratios $Z_\tau/Z_t$ are well-defined for any $t\in\NN$ and any finite stopping time $\tau\ge t$. Note that each $Z_t$ is defined up to $\Rtt$-nullsets, and therefore also up to $\Ptt$-nullsets.
    
    \item If $\Ptt\ll_\text{loc}\Rtt$ has likelihood ratio process $(Z_t)$, the following `Bayes formula' holds: for any $t\in\NN$, any finite stopping times $\tau$, and any nonnegative $\Fcal_\tau$-measurable random variable $Y$, one has
    \[
    \EE_\Ptt[Y\mid\Fcal_t]= \EE_\Rtt\left[\left.\frac{Z_\tau}{Z_t}Y\right|\Fcal_t\right] \1_{\{Z_t > 0\}}, \qquad \text{ $\Ptt$-almost surely. }
    \]
    The right-hand side is uniquely defined $\Rtt$-almost surely (not just $\Ptt$-almost surely), and therefore provides a `canonical' version of $\EE_\Ptt [Y\mid\Fcal_t]$. \emph{We always use this version.} This allows us to view such conditional expectations under $\Ptt$ as being well-defined up to $\Rtt$-nullsets.
\end{itemize}

One might ask why we work with \emph{local} domination, rather a `global' condition like $\Ptt \ll \Rtt$ for all $\Ptt$ of interest. The answer is that such a condition would be far too restrictive, as we now illustrate. Let $(X_t)_{t\in\NN}$ be a sequence of random variables. For each $\eta\in\mathbb R$, let $\Ptt^\eta$ be the distribution such that the $X_t$ become i.i.d.\ normal with mean $\eta$ and unit variance. By the strong law of large numbers, $\Ptt^\eta$ assigns probability one to the event $A^\eta :=\{\lim_{t\to\infty}t^{-1}\sum_{s=1}^t X_s=\eta\}$. Moreover, the events $A^\eta$ are mutually disjoint: $A^\eta \cap A^\nu = \emptyset$ whenever $\eta\ne\nu$. This means by definition that the measures $\Ptt^\mu$ are all mutually singular. Since there is an uncountable number of them, there cannot exist a measure $\Rtt$ such that $\Ptt^\eta\ll\Rtt$ for all $\eta$. On the other hand, if $\Ptt^\eta|_t$ denotes the law of the partial sequence $X_1,\ldots,X_t$ for some $t \in \NN$, then the measures $\Ptt^\eta|_t$, $\eta\in\mathbb R$, are all mutually absolutely continuous. In particular, we could (for instance) use $\Rtt=\Ptt^0$ as reference measure and obtain $\Ptt^\eta \ll_{\text{loc}} \Rtt$ for all $\eta\in\mathbb R$.

\subsection{Essential supremum}
\label{sec:esssup}

On some probability space, consider a collection $(Y_\alpha)_{\alpha\in\Acal}$ of random variables, where $\Acal$ is an arbitrary index set. If $\Acal$ is uncountable, the pointwise supremum $\sup_{\alpha\in\Acal}Y_\alpha$ might not be measurable (not a random variable). Moreover, it might happen that $Y_\alpha=0$ almost surely for every $\alpha\in\Acal$, but $\sup_{\alpha\in\Acal}Y_\alpha=1$. For these reasons, the pointwise supremum is often not useful. Instead, one can use the \emph{essential supremum}.

\begin{proposition}
    There exists a $[-\infty,\infty]$-valued random variable $Y$, called the \emph{essential supremum} and denoted by $\esssup_{\alpha\in \Acal}Y_\alpha$, such that
    \begin{enumerate}
        \item $Y\ge Y_\alpha$, almost surely, for every $\alpha\in \Acal$,
        \item if $Y'$ is a random variable that satisfies $Y'\ge Y_\alpha$, almost surely, for every $\alpha\in \Acal$, then $Y'\ge Y$, almost surely.
    \end{enumerate}
    The essential supremum is almost surely unique.
\end{proposition}

In words, the essential supremum is the smallest almost sure upper bound on $(Y_\alpha)$. The proposition guarantees that it always exists. In some cases, more can be said: the essential supremum can be obtained as the limit of an increasing sequence.

\begin{proposition}\label{P_esssup_closed_max}
    Suppose $(Y_\alpha)$ is closed under maxima, meaning that for any $\alpha,\beta\in \Acal$ there is some $\gamma\in \Acal$ such that $Y_\gamma=\max\{Y_\alpha,Y_\beta\}$. Then there is a sequence $(\alpha_n)$ such that $(Y_{\alpha_n})$ is an increasing sequence and $\esssup_{\alpha\in \Acal}Y_\alpha = \lim_{n \to \infty}  Y_{\alpha_n}$.
\end{proposition}

For more information about the essential supremum (and infimum), as well as proofs of the above results, we refer to Section~A.5 in \cite{MR2169807}.

\section{Omitted proofs}   \label{A:proofs}

\begin{proof}[Proof of Proposition~\ref{P:210407}]
    It is clear that \ref{P:210407.2} implies \ref{P:210407.1}. Now assume that \ref{P:210407.1} holds, for the moment with $k=1$.
    Let $M$ denote the transition matrix of the Markov process $(X_t)$, i.e., the matrix with elements $p_{j|i}$ for $i, j \in \{0, 1
    \}$. 
    Note that exchangeability yields
    \[
        \Ptt(X_3 = j | X_2 = i) = \Ptt(X_3 = j | X_1 = i) 
    \]
    for $i,j \in \{0, 1\}$, hence $M^2 = M$.  If $M$ has full rank then $M$ is the identity matrix and $(X_t)$ is constant in time. If $M$ does not have full rank then its two rows are the same since its row sums are always equal to one. It follows that $X_1,X_2,\ldots$ are independent. Moreover, by exchangeability $X_1, X_2,X_3,\ldots$ have the same distribution. Thus $(X_t)$ is i.i.d. This shows the statement for $k = 1$.

    Let us assume that \ref{P:210407.1} holds for a  general $k \in \mathbb{N}$. Instead of providing a modification of the above argument now with transition tensors, let us give an alternative probabilistic, more verbose argument. For simplicity, we only focus on the case $k=2$; it is clear how to generalize this argument. Note that the Markov property yields that the law of $X_5$ conditional on $(X_1, X_2, X_3, X_4)$ equals the law of $X_5$ conditional on $(X_3, X_4)$. By exchangeability, this also equals the law of $X_5$ conditional on $(X_1, X_2)$. This again yields that either $(X_t)$ is constant or the law of $X_5$ does not depend on the earlier values $(X_1, X_2, X_3, X_4)$. This observation concludes the proof.  
\end{proof}

\begin{remark}
    In this remark, let us briefly discuss Proposition~\ref{P:210407} in the context of a general Markov process $(Y_t)$, say with a countable state space. This is not required below but sheds additional light on  Proposition~\ref{P:210407}\ref{P:210407.2}. To this end, let us hence assume that $(Y_t)$ is an exchangeable Markov process with countable state space.  First of all, note that exchangeability yields that $(Y_t)$ reaches state $i$ from state $j$ if and only if it reaches state $j$ from state $i$. Indeed, for $s, t \in \mathbb{N}$ we have $\Ptt(Y_s = i, Y_t = j) = \Ptt(Y_s = j, Y_t = i)$. Hence, the state space can be partitioned in subsets, say with index set $\mathcal{I}$, such that in each of these subsets, each state `communicates' with any other. A suitable modification of the proof of  Proposition~\ref{P:210407} now shows that $(Y_t)$ can be constructed as follows. First draw an $\mathcal{I}$-valued random variable to choose the subset of the state space in which $(Y_t)$ will take values, and then choose $i.i.d.$~draws of a distribution whose support is exactly this subset. As a corollary, conditionally on $Y_1$, the random variables $Y_2, Y_3, \ldots$ are i.i.d. Indeed, it is easy to see that a $\{0, 1\}$-valued process $(X_t)$ that satisfies Proposition~\ref{P:210407}\ref{P:210407.2} also satisfies that  $X_2, X_3, \ldots$ are i.i.d.~conditional on $X_1$.
\end{remark}

\begin{proof}[Proof of Theorem~\ref{thm:P-Snell}]
The proof is essentially a simplified version of an argument due to Delbaen \cite[Theorem~11]{MR2276899}. This result is argued in continuous time and on a bounded time interval. For the convenience of the reader, we provide a self-contained proof for this paper's discrete-time, infinite-horizon setup. 

For each fixed $s \in \NN$, $L_s$ is defined as the essential supremum of the family consisting of all $\EE_\Qtt[E_\tau\mid\Fcal_s]$, indexed by all pairs $(\Qtt,\tau)$ with $\Qtt\in\Qcal$ and $\tau\ge s$ a finite stopping time. Here we use a version of the conditional expectation that satisfies $\EE_\Qtt[E_\tau\mid\Fcal_s]=0$ on the event $\{Z_s = 0\}$, where $(Z_t)$ denotes the likelihood ratio process of $\Qtt$ (see Appendix~\ref{sec:technical_appendix}). We claim that this family of conditional expectations is closed under maxima. To prove this claim, let $(\Qtt,\tau)$ and $(\Qtt',\tau')$ be given. Let $A:=\{
\EE_\Qtt[E_\tau\mid\Fcal_s] \ge \EE_{\Qtt'}[E_{\tau'}\mid\Fcal_s]\}$ and set
\(
\tau'' := \tau\mathbf1_A + \tau'\mathbf1_{A^c} 
\)
and
\[
Z''_t := \begin{cases}
Z_t, & t \le s \\
\mathbf1_A Z_t + \mathbf1_{A^c} Z_s \dfrac{Z'_t}{Z'_s}, & t>s 
\end{cases}
\]
where $(Z_t)$ and $(Z'_t)$ are the likelihood ratio processes of $\Qtt$ and $\Qtt'$, respectively. Note that $Z_s' > 0$ on $A^c$ so that $Z''_t$ is well-defined. Since $A$ belongs to $\Fcal_s$ and $\tau,\tau'\ge s$, $\tau''$ is a (finite) stopping time. Moreover, since $\Qcal$ is fork-convex, $(Z''_t)$ is the likelihood ratio process of some $\Qtt''\in\Qcal$. We now compute
\begin{align*}
   \EE_{\Qtt''}[E_{\tau''}\mid\Fcal_s] &=\EE_{\Qtt''}[\mathbf1_A E_{\tau}\mid\Fcal_s] +\EE_{\Qtt''}[\mathbf1_{A^c}E_{\tau'}\mid\Fcal_s] \\
    &=\EE_\Rtt\left[\left.\frac{Z''_\tau}{Z''_s}\mathbf1_A E_{\tau}\right|\Fcal_s\right] +\EE_\Rtt\left[\left.\frac{Z''_{\tau'}}{Z''_s}\mathbf1_{A^c} E_{\tau'}\right|\Fcal_s\right] \\
    &=\EE_\Rtt\left[\left.\frac{Z_\tau}{Z_s}\mathbf1_A E_{\tau}\right|\Fcal_s\right] +\EE_\Rtt\left[\left.\frac{Z'_{\tau'}}{Z'_s}\mathbf1_{A^c} E_{\tau'}\right|\Fcal_s\right] \\
    &= \mathbf1_A \EE_\Qtt\left[\left. E_{\tau}\right|\Fcal_s\right] + \mathbf1_{A^c} \EE_{\Qtt'}\left[ E_{\tau'}\mid\Fcal_s\right] \\
    &= \max\left\{\EE_\Qtt\left[ E_{\tau}\mid\Fcal_s\right],\EE_{\Qtt'}\left[ E_{\tau'}\mid\Fcal_s\right]\right\}.
\end{align*}
This demonstrates closure under maxima.

Now fix any $\Qtt\in\Qcal$ and $s\in\NN$. Thanks to the closure property under maxima, Proposition~\ref{P_esssup_closed_max} shows that there exist families $(\Qtt_n)$ of measures in $\Qcal$ and $(\tau_n)$ of finite stopping times taking values in $\{s, s+1, \ldots\}$ such that $\EE_{\Qtt_n}[E_{\tau_n}\mid\Fcal_s]\uparrow L_s$ almost surely under $\Rtt$, and hence under $\Qtt$. Therefore, by the conditional version of the monotone convergence theorem,
\begin{equation}\label{eq:Snell_proof_eq1}
   \EE_\Qtt[L_s\mid\Fcal_{s-1}] =\EE_\Qtt\left[\left.\lim_{n \to \infty} \EE_{\Qtt_n}[ E_{\tau_n}\mid\Fcal_s]\right|\Fcal_{s-1}\right] = \lim_{n \to \infty} \EE_\Qtt[\EE_{\Qtt_n}[ E_{\tau_n}\mid\Fcal_s]\mid\Fcal_{s-1}].
\end{equation}
Replacing $\Qtt_n$ by $(1-n^{-1})\Qtt_n + n^{-1}\Qtt$ we still have \eqref{eq:Snell_proof_eq1} and, in addition, $\Qtt$ absolutely continuous with respect to $\Qtt_n$. From now on we use this modified choice of $\Qtt_n$. Let $(Z_t)$ and $(Z^n_t)$ be the likelihood ratio processes of $\Qtt$ and $\Qtt_n$, respectively, and define
\[
\widetilde Z^n_t := \begin{cases}Z_t,& t \le s, \\ Z_s\dfrac{Z^n_t}{Z^n_s}, &t > s. 
\end{cases}
\]
By fork-convexity, $(\widetilde Z^n_t)$ is the likelihood ratio process of some $\widetilde \Qtt_n\in\Qcal$. We then get
\begin{align*}
   \EE_\Qtt[\EE_{\Qtt_n}[ E_{\tau_n}\mid\Fcal_s]\mid\Fcal_{s-1}] &=\EE_\Rtt\left[\left.\frac{Z_s}{Z_{s-1}}\EE_\Rtt\left[\left.\frac{Z^n_{\tau_n}}{Z^n_s} E_{\tau_n}\right|\Fcal_s\right]\right|\Fcal_{s-1}\right] \\
    &=\EE_\Rtt\left[\left.\frac{Z_s}{Z_{s-1}}\frac{Z^n_{\tau_n}}{Z^n_s} E_{\tau_n}\right|\Fcal_{s-1}\right] \\
    &=\EE_\Rtt\left[\left.\frac{\widetilde Z^n_{\tau_n}}{\widetilde Z^n_{s-1}} E_{\tau_n}\right|\Fcal_{s-1}\right] \\
    &=\EE_{\widetilde \Qtt_n}\left[ E_{\tau_n}\mid\Fcal_{s-1}\right] \quad
    \le L_{s-1}.
\end{align*}
Combining this with \eqref{eq:Snell_proof_eq1} gives $\EE_\Qtt[L_s\mid\Fcal_{s-1}] \le L_{s-1}$. Iterating this inequality and using that $\Fcal_0$ is trivial yields $\EE_\Qtt[L_s]\le L_0$. 
In particular, $L_s$ is $\Qtt$-integrable.
Since $(E_t)$ is a $\Qcal$-safe e-process, we have $L_0=\sup_{\Qtt\in\Qcal,\,\tau\ge0}\EE_\Qtt[E_\tau]\le1$.  Since $\Qtt\in\Qcal$ and $s\in\NN$ were arbitrary, this proves that $(L_t)$ is a $\Qcal$-NSM with $L_0\le1$.

Let $(L'_t)$ be another $\Qcal$-NSM that dominates $(E_t)$. Then for any $\Qtt\in\Qcal$, any $t\in\{0,1,\ldots\}$, and any finite stopping time $\tau\ge t$, the optional stopping theorem under $\Qtt$ gives $L'_t\ge\EE_\Qtt[L'_\tau\mid\Fcal_t] \ge\EE_\Qtt[E_\tau\mid\Fcal_t]$. Therefore $L'_t\ge L_t$ by the definition of essential supremum.
\end{proof}

\end{document}